\documentclass[3p,11pt]{elsarticle}
\usepackage{amsmath}
\usepackage{graphicx}
\usepackage{caption} 
\usepackage{wrapfig}
\usepackage{color,soul}
\usepackage{enumitem}
\usepackage{subcaption}
\usepackage{comment}
\usepackage{amsmath}
\usepackage{times}
\usepackage{amssymb}
\usepackage{array}
\usepackage{yhmath}
\usepackage{latexsym}
\usepackage{amsthm}
\usepackage{amsfonts}
\usepackage{nomencl}
\usepackage[table]{xcolor}
\usepackage{lineno}

\makenomenclature
\usepackage{hyperref}
\usepackage{array,tabularx}
\usepackage{scalerel,stackengine}
\usepackage{graphics}
\usepackage[english]{babel}
\allowdisplaybreaks
\newtheorem{theorem}{Theorem}[section]
\newtheorem{corollary}{Corollary}[theorem]
\newtheorem{lemma}[theorem]{Lemma}

\stackMath
\numberwithin{equation}{section}
\date{}
\usepackage[utf8]{inputenc}

\begin{document}
\begin{frontmatter} 
\title{Sonic-supersonic jet flows from a straight two-dimensional nozzle with van der Waals equation of statee \tnoteref{mytitlenote}} 
\author[1]{Anamika Pandey}
\author[1]{T. Raja Sekhar}
\address[1]{Department of Mathematics, Indian Institute of Technology Kharagpur, Kharagpur, West Bengal, India} 
\cortext[mycorrespondingauthor]{Corresponding author} 
\ead{trajasekhar@maths.iitkgp.ac.in} 

\begin{abstract}
This article concerns sonic-supersonic jet flows issuing from a two-dimensional straight nozzle described by the steady compressible Euler system under the van der Waals equation of state. The flow state is prescribed at the nozzle exit, while the surrounding medium is assumed to be either a vacuum or a static atmosphere with lower pressure. When the flow reaches the sonic state at the nozzle exit, the governing equations become degenerate hyperbolic and the resulting jet flow problem can be formulated as a degenerate free boundary problem. The analysis is further complicated by singular behavior near the endpoints of the nozzle exit.  By employing characteristic decompositions and suitable a priori estimates, we establish the global existence of locally Lipschitz continuous sonic-supersonic jet flows expanding into a vacuum. Moreover, in the presence of a static atmosphere with pressure lower than that at the nozzle exit, we prove the local existence of sonic-supersonic jet flows near the nozzle exit.
\end{abstract}

\end{frontmatter}

\section{Introduction}\label{Introduction}

Jet flows issuing from nozzles constitute a fundamental class of problems in compressible fluid dynamics and have long attracted considerable attention due to their importance in both practical applications and mathematical theory. Such flows arise naturally in a wide range of engineering settings, including aerospace propulsion systems, rocket engines, jet technologies and high-speed aerodynamic devices. Depending on the pressure conditions inside and outside the nozzle, the exhausting gas may undergo significant acceleration and expansion, leading to the formation of intricate flow patterns involving rarefaction waves, compression waves, shock fronts, contact discontinuities and nonlinear wave interactions. In particular, when a jet expands into a surrounding medium of lower pressure or into a vacuum, the interface separating the jet from the exterior environment is generally unknown in advance and must be determined simultaneously with the flow field. Consequently, the mathematical description of nozzle jet flows naturally gives rise to free boundary problems for nonlinear systems of partial differential equations. From the analytical viewpoint, jet flows provide a rich framework for investigating fundamental questions concerning the existence, uniqueness, regularity and qualitative behavior of solutions to the compressible Euler equations. The interaction of nonlinear waves, the appearance of transonic phenomena and the presence of free boundaries often generate significant mathematical difficulties, especially when the flow approaches sonic states where the governing equations may change type or become degenerate. Understanding these phenomena is therefore of considerable interest not only for applications in gas dynamics but also for the broader development of the theory of nonlinear hyperbolic equations and free boundary problems arising in compressible fluid flows.

Among the various flow regimes arising in nozzle flows, transonic phenomena occupy a particularly important position due to the interplay between subsonic, sonic and supersonic states. Of special interest is the sonic transition, which occurs when the flow speed reaches the local sound speed and the governing equations undergo a fundamental change in character. In many nozzle configurations, the sonic state serves as a critical threshold separating qualitatively different flow regimes and plays a decisive role in determining the subsequent evolution of the flow. From the mathematical perspective, the presence of sonic curves introduces significant analytical difficulties because the Euler equations become degenerate at these locations, leading to the loss of strict hyperbolicity and the breakdown of standard techniques applicable to uniformly hyperbolic systems. Moreover, sonic transitions are often accompanied by the emergence of complex wave structures, free boundaries and singular behaviors that strongly influence the global flow pattern. Consequently, the study of sonic and sonic-supersonic flows has become a central topic in the theory of compressible Euler equations, motivating extensive research on degenerate hyperbolic equations, transonic free boundary problems and the qualitative properties of solutions near sonic regions.

Motivated by both its practical significance and its rich mathematical structure, the nozzle jet flow problem has been investigated extensively over the past several decades. A pioneering contribution was made by Courant and Friedrichs \cite{Courant1948} in their classical monograph ``Supersonic Flow and Shock Waves'', where a detailed description of various jet flow patterns generated by supersonic flows exhausting from two-dimensional nozzles was presented. Since then, considerable efforts have been devoted to the rigorous mathematical analysis of nozzle jet flows under different physical assumptions and geometric configurations. In particular, Chen and Qu \cite{chen2010} investigated supersonic jet flow solution expanding into a lower-pressure atmosphere. They established the existence of classical solutions in both the interaction region of steady rarefaction waves and the reflection region near the jet interface for sufficiently small pressure differences, while also analyzing the formation of a vacuum for larger pressure differences. Subsequently, \cite{chenQu2014} established the global existence of continuous and piecewise smooth supersonic flows in two-dimensional convex ducts by analyzing the interaction of rarefaction waves and further characterized the occurrence of vacuum. Similar ideas were later extended to relativistic Euler flows in convex ducts in \cite{Luan2018}. More recently, Lai \cite{Lai2023} established the existence of sonic--supersonic jet flows issuing from a 2-D nozzle. Besides the aforementioned works, compressible flows in divergent nozzle configurations have also been widely studied \cite{Yuanyuan2024, Xiaomin2026, laigeng226, park2026}. Related investigations include nozzle flow problems in finitely long nozzles \cite{zhang2026, Gao2025, yang2022} and smooth supersonic flows in multi-dimensional nozzles \cite{Chen2014, genglai2020, gangxu2021}, along with the references therein. 
 
In addition to nozzle jet flows, substantial progress has been made in the study of degenerate hyperbolic boundary value problems arising in compressible fluid dynamics. Local smooth supersonic solutions near a prescribed sonic curve for the steady isentropic Euler system in two space dimension were derived in \cite{Zhang2014}. Subsequently, \cite{yanbo2020} generalized this analysis to the two-dimensional steady compressible Euler equations and established the existence of classical sonic-supersonic solutions near a sonic curve. Smooth transonic flow phenomena associated with Keldysh-type degenerate hyperbolic boundaries and related transonic free boundary problems has also been investigated in various settings \cite{wangc2019}, revealing the intricate mathematical structure generated by sonic degeneracy. Beyond the works mentioned above, extensive theoretical and numerical studies have also been carried out on the compressible Euler equations and related nonlinear hyperbolic systems in various physical and mathematical settings; see, for example, \cite{Manganaro2019, Manganaro2017, Mzafar2026, Pandey1, Pandey2, Rahul2026} and the references cited therein.

Although substantial progress has been made in the analysis of sonic-supersonic jet solutions, most of the available results are established within the framework of ideal or polytropic gas laws. The analysis in these settings relies heavily on the particular structure of the pressure-density relation and the associated characteristic geometry. When a non-ideal equation of state such as the van der Waals law is considered, the nonlinear features of the governing equations become significantly more intricate, leading to substantial modifications in the characteristic structure and sonic degeneracy. As a consequence, several analytical techniques developed for the polytropic case cannot be applied directly and a new analysis is required to construct sonic-supersonic jet flow solutions in the non-ideal gas setting.

In this article, we investigate sonic-supersonic jet flows issuing from a 2-D straight nozzle for compressible fluids governed by the van der Waals equation of state. The flow is assumed to attain the sonic state at the nozzle exit and subsequently expand into either a vacuum or a static atmosphere of lower pressure. In contrast to the uniformly supersonic case, the presence of sonic states leads to the degeneration of the governing equations and gives rise to a degenerate hyperbolic free boundary problem. Moreover, singular behavior occurs at the nozzle exit points $\mathcal{A}_1$ and $\mathcal{A}_2$ located at $(0,-l)$ and $(0,l)$, respectively, which further complicates the analysis. While substantial progress has been made in the study of sonic-supersonic jet flows for ideal and polytropic gases, extending these results to non-ideal gases presents significant mathematical challenges due to the more intricate nonlinear structure induced by the van der Waals pressure law. The objective of this work is to establish the existence of sonic-supersonic jet flow solutions in the van der Waals setting and thereby extend the existing theory beyond the classical polytropic framework.

The widely studied two-dimensional steady Euler equations read \cite{Courant1948}
\begin{equation}\label{Euler}
\begin{split}
    &\frac{\partial(\rho u) }{\partial x}+\frac{\partial (\rho v)}{\partial y}=0,\\
    &\frac{\partial(\rho u^2+p) }{\partial x}+\frac{\partial (\rho uv)}{\partial y}=0,\\
    &\frac{\partial(\rho uv) }{\partial x}+\frac{\partial (\rho v^2+p)}{\partial y}=0, \; (x,y) \in \mathbb{R}^2,
\end{split}
\end{equation}
where the variables $u, v$ represent x- and y-direction flow velocity, $\rho$ is density and $p=p(\tau)$ is the pressure and $\tau=\frac{1}{\rho}$ denotes the specific volume. 

Although the preceding analysis is presented for the classical polytropic gas, the ideal gas law is often inadequate for describing the behavior of real gases under high-pressure or low-temperature conditions. Such non-ideal effects arise in a wide range of engineering applications, including internal combustion engines, refrigerators, heat pumps, turbogenerators, hot-air balloons, Organic Rankine Cycle (ORC) systems, gas storage facilities, fire extinguishers and many other practical applications \cite{Mzafar2016}. The principal limitation of the polytropic gas model is that it neglects both the intermolecular attractive forces and the finite volume occupied by gas molecules. To account for these non-ideal effects, we consider a compressible fluid governed by the polytropic van der Waals equation of state \cite{callen1998thermodynamics, Pandey3},
\begin{equation}\label{vanderwaals}
p(\tau)=\frac{K}{(\tau-b)^{\gamma+1}}-\frac{a}{\tau^2},
\end{equation}
where $K>0$ is a constant related to the entropy of the gas and $\gamma\in(0,1)$. The positive constants $a$ and $b$ characterize the intermolecular attractive forces and the finite volume occupied by gas molecules, respectively. In the special case $a=0$, the above equation reduces to the dusty gas model, whereas $a=b=0$ yields the classical polytropic ideal gas law.

According to \cite{MR3820376}, the van der Waals equation of state can be classified into the following three categories, as illustrated in Figure~\ref{vanderwaalsgas}.
\begin{itemize}
\item $p'(\tau)<0$ and $p''(\tau)>0$.
\item $p'(\tau)<0$ for $\tau\in(b,\infty)$, while $p''(\tau)>0$ for $\tau\in(b,\tau_1)\cup(\tau_2,\infty)$ and $p'(\tau)=p''(\tau)=0$ for $\tau\in(\tau_1,\tau_2)$.
\item $p'(\tau)<0$ for $\tau\in(b,\infty)$, whereas $p''(\tau)>0$ for $\tau\in(b,\tau_1)\cup(\tau_2,\infty)$, $p''(\tau)<0$ for $\tau\in(\tau_1,\tau_2)$ and $p''(\tau_1)=p''(\tau_2)=0$.
\end{itemize}

The van der Waals gas is known to possess regions where the pressure law loses convexity, leading to additional mathematical difficulties associated with changes in the wave structure and loss of genuine nonlinearity. In the present work, we restrict our attention to a physically relevant regime in which the pressure function remains strictly convex. This assumption guarantees the hyperbolicity of the Euler system and allows the characteristic structure of the flow to be analyzed within a stable analytical framework. Moreover, many of the arguments developed in this article rely only on the monotonicity and convexity properties of the pressure law and are therefore applicable to a broader class of convex equations of state beyond the van der Waals model; see, for example, \cite{MR3342411, MR3434409}. Accordingly, throughout this article, the pressure function $p=p(\tau)$ is assumed to satisfy
\begin{equation*}
    p'(\tau)<0, \; p''(\tau)>0,\; \kappa(\tau)>0,\;m(\tau)>0 \;\text{for} \; \tau>\tau_1.
\end{equation*}
For technical convenience, we further assume that $\kappa'(\tau)>0$. This assumption is justified for sufficiently large values of $\tau$. Indeed, introducing the dimensionless variables
$$\Pi=\frac{b^2p(\tau)}{a}, \qquad \bar{\tau}=\frac{\tau}{b},$$
the equation of state \eqref{vanderwaals} can be written in the dimensionless form
$$\Pi(\bar\tau)=\frac{A'}{(\bar\tau-1)^{\gamma+1}}-\frac{1}{\bar\tau^2},$$ where $A'=\frac{Ab^{1-\gamma}}{a}$, see \cite{MR3434409}. Without loss of generality, we normalize $A'=1$. A direct calculation then shows that
$$\kappa'(\tau)=\frac{\kappa(\tau)^2}{2b}\frac{d}{d\bar\tau}\left(\frac{\bar\tau\Pi''(\bar\tau)}{\Pi'(\bar\tau)}\right)$$ and
$$\frac{d}{d\bar\tau}\left(\frac{\bar\tau\Pi''(\bar\tau)}{\Pi'(\bar\tau)}\right)=\frac{H_\gamma(\bar\tau)}{[2(\bar\tau-1)^{\gamma+1}-(\gamma+1)\bar\tau^3(\bar\tau-1)]^2},$$
where $H_\gamma(\bar\tau)=(\gamma+1)^2(\gamma+2)\bar\tau^6-2\bar\tau^2(\bar\tau-1)^{\gamma+2}(\gamma+1)[(\gamma+1)^2\bar\tau^2+(7\gamma+5)\bar\tau+9]$.
Moreover, $H_0(\bar{\tau})>0$ for $\bar{\tau}>4$. By continuity of $H_\gamma(\bar{\tau})$ with respect to $\gamma$ in a neighborhood of $\gamma=0$, it follows that $\kappa'(\tau)>0$ for sufficiently large $\tau>4b$.

\begin{figure}[htbp]
\centering

\begin{subfigure}{0.32\textwidth}
    \centering
    \includegraphics[width=0.65\linewidth]{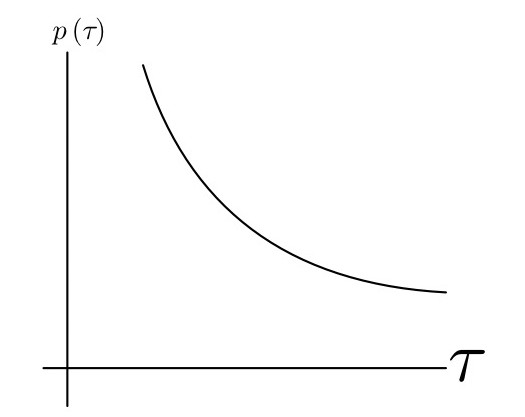}
    \label{fig:sfig1}
\end{subfigure}
\hfill
\begin{subfigure}{0.32\textwidth}
    \centering
    \includegraphics[width=0.9\linewidth]{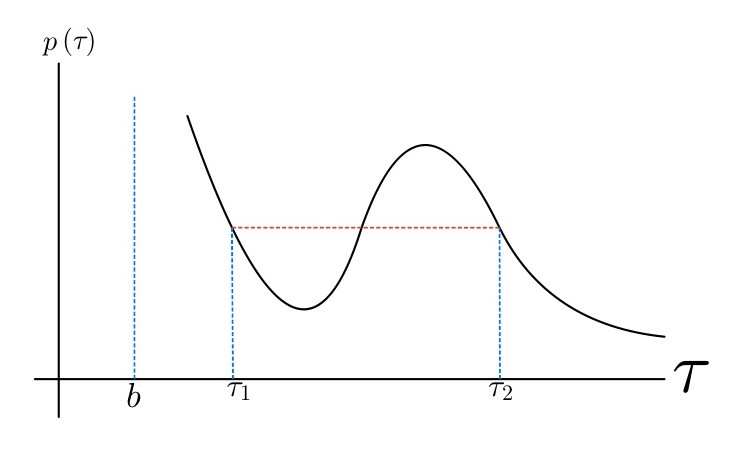}
    \label{fig:sfig2}
\end{subfigure}
\hfill
\begin{subfigure}{0.32\textwidth}
    \centering
    \includegraphics[width=0.9\linewidth]{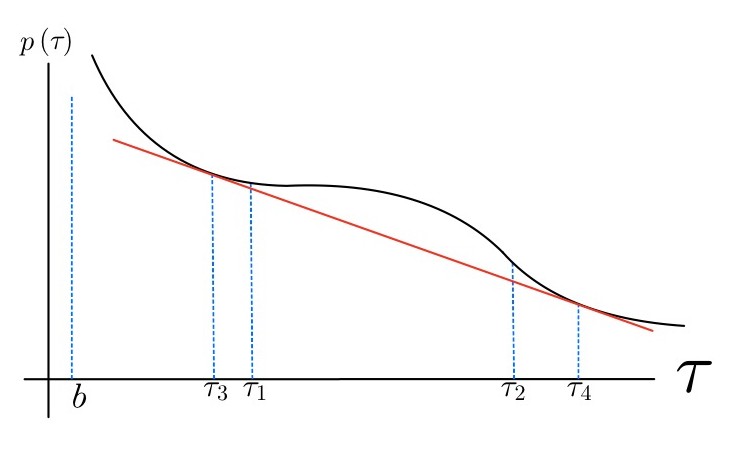}
   \label{fig:sfig3}
\end{subfigure}
\caption{Three possible configurations of the van der Waals pressure function $p(\tau)$ corresponding to different convexity properties.}
\label{vanderwaalsgas}
\end{figure}

To analyze the degenerate hyperbolic free boundary value problem arising from sonic-supersonic jet flows, we employ the characteristic decomposition approach which was originally introduced in the study of simple wave solutions for the two-dimensional Euler equations of gas dynamics \cite{Jiequan2006}, has subsequently become an effective tool in the investigation of multidimensional wave interactions, transonic phenomena and related free boundary problems \cite{Yfan2023, Jchen2025, Rahul3}. By exploiting the underlying characteristic framework of the Euler system and establishing suitable a priori estimates, we are able to handle both the sonic degeneracy and the singular behavior at the endpoints of the nozzle exit. As a consequence, in the case of vacuum we develop the existence of global locally Lipschitz continuous jet solutions for sonic-supersonic flow. Furthermore, when the jet exhausts into a static atmosphere with pressure $0\leq p_{atm}<p_0$, we prove the local existence of sonic-supersonic jet flow solutions for $u_0=c_0$.

The remainder of this article is arranged as follows: In Section 2, we reformulate the problem together with the statement of the main results and derive several characteristic identities and decompositions that play a fundamental role in the subsequent analysis. These decompositions are subsequently employed to construct the centered simple wave solutions and derive the necessary a priori estimates, leading to the proof of the global existence of supersonic and sonic-supersonic jet flow solutions expanding into a vacuum in Section 3. These decompositions are subsequently employed to construct the centered simple wave solutions, establish the corresponding a priori estimates and prove the global existence of supersonic and sonic-supersonic jet flow solutions expanding into a vacuum in Section 3. In Section 4, we consider the case where the surrounding medium is a static atmosphere and prove the local existence of a sonic-supersonic jet solution at the nozzle exit. The paper concludes with Section 5.

\section{Preliminary analysis and problem description}

\subsection{Problem formation and main results}

In this subsection, we formulate the jet flow problem and state the main results of this article. Consider a two-dimensional straight nozzle that is symmetric about the $x$-axis and enclosed by the nozzle walls
\begin{equation*}\label{walls}
      \mathcal W_1=\{(x,y):\, x\in(-\infty,0],\;y= -l\},\quad \mathcal W_2=\{(x,y):\, x\in(-\infty,0],\;y= l\}.
\end{equation*}
The exit section of the nozzle is positioned on the $y$-axis. Consider a uniform flow having pressure $p_0$, velocity $(u_0,0)$, density $\rho_0$ and sound speed $c_0$ arrives at the nozzle exit, where $u_0\geq c_0$. 

We first examine the scenario where the nozzle is surrounded by a vacuum. Then the flow exhausts from the nozzle into the vacuum and is separated from the vacuum region by a vacuum boundary; see Figure \ref{vacuumboundary}. To describe the resulting sonic-supersonic jet flow, we consider system \eqref{Euler} together with the boundary conditions
\begin{equation}\label{Boundary_v}
\left.(c,u,v)\right|_{\overline{\mathcal{A}_1\mathcal{A}_2}}=(c_0, u_0,0),\qquad
\left.c\right|_{x=\varphi(y),\,|y|>l}=0,\qquad
\varphi(\pm l)=0,
\end{equation}
where for $|y|>l$ the curve $x=\varphi(y)$ denotes the vacuum boundary. 

The expansion of compressible jet flows into a vacuum has attracted considerable attention in both theoretical and applied fluid dynamics; see, for instance, \cite{MR3342411, MR3434409, tuer1974axisymmetric, cassanova1967expansion}. From a physical perspective, such configurations arise naturally in high-altitude and near-space environments, where the surrounding medium is sufficiently rarefied that vacuum effects become significant. In the present setting, the vacuum boundary is not known a priori and must be determined as part of the solution, which leads to a free BVP (Boundary value problem)  for the Euler system. Such problems are of considerable basic and practical interest. The assumption that the surrounding medium is vacuum is natural if the aerospace vehicle is in the outer atmosphere or the surrounding atmosphere is very rare.

We investigate \eqref{Euler}--\eqref{Boundary_v} in the sonic case
 $u_0=c_0$. Since the incoming flow is sonic, the Euler system \eqref{Euler} becomes degenerate hyperbolic at the nozzle exit. As a result, the free boundary value problem \eqref{Euler}--\eqref{Boundary_v} is of degenerate hyperbolic type. Nevertheless, owing to the absence of compression in the jet, we establish the existence of a global solution. The corresponding result is summarized in the following theorem:
\begin{theorem}\label{main1}
Assume that the incoming flow is either sonic or supersonic, i.e.,
$u_0\ge c_0$ and that the pressure is governed by the van der Waals equation of state \eqref{vanderwaals}. Then the free boundary value problem \eqref{Euler}--\eqref{Boundary_v} admits a solution
\[
(c,u,v)\in C^{0,1}\!\left(\Sigma(M_0)\right)
\cap
C^{0}\!\left(\overline{\Sigma(M_0)}\right),
\]
where
\[
\Sigma(M_0)=\{(x,y)\mid x>\varphi(y;M_0),\,-\infty<y<+\infty\},
\]
with
\[
\varphi(y;M_0)=
\begin{cases}
(|y|-l)\cot\alpha_v, & |y|\ge l,\\
0, & |y|<l,
\end{cases}
\qquad
M_0=\frac{u_0}{c_0},
\]
where $\alpha_v$ denotes the flow angle along the vacuum boundary and is given by
\[
\alpha_v=\int_{\tau_0}^{\infty} \frac{\sqrt{-p'(\tau)}\sqrt{u^2+v^2+\tau^2p'(\tau)}}{u^2+v^2}d\tau.
\]
Moreover, $c$ remains positive in $\Sigma(M_0)$ and vanishes on the vacuum boundary
\[
x=\varphi(y;M_0),\qquad |y|>l.
\]
\end{theorem}
We further investigate the case where the nozzle exhausts into a static atmosphere with pressure $p_{atm}$, satisfying $0\leq p_{atm}<p_0$, where $p_0$ is the pressure of the supersonic flow at the nozzle exit. The contact discontinuity is formed by the jet boundary for separating the exhaust gas from the surrounding atmosphere; see Figure \ref{Jetboundary}. Denoting the free jet boundary by $\overline{\mathcal{A}_1\mathcal{A}_2}=\{(x,y) \;| \; x=0,\; |y|<l\}$ and $x=\psi(y), |y|>l$, the flow is governed by the system \eqref{Euler} together with the boundary conditions
\begin{figure}
\begin{subfigure}{.5\textwidth}
  \centering
  \includegraphics[width=.7\linewidth]{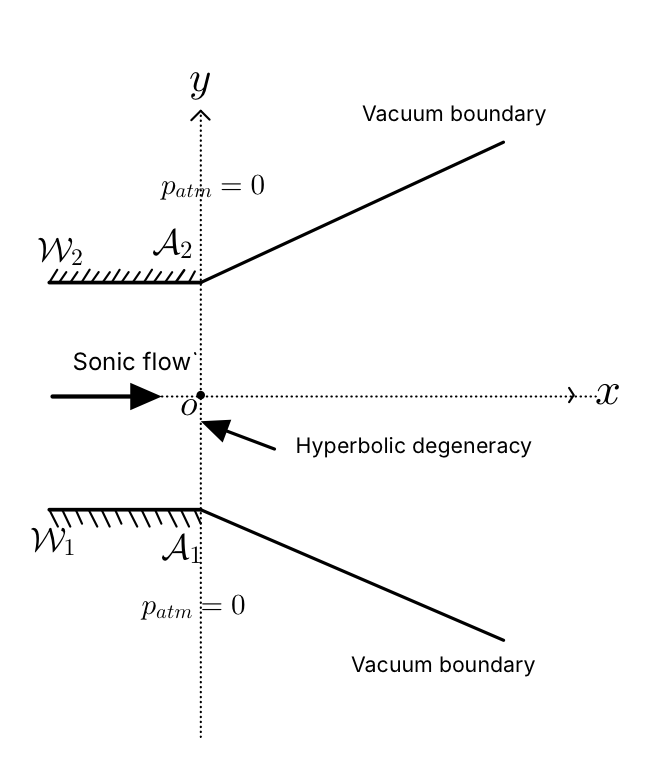}
  \caption{Vacuum boundary; $p_{atm}=0$}  
  \label{vacuumboundary}
\end{subfigure}%
\begin{subfigure}{.5\textwidth}
  \centering
  \includegraphics[width=.7\linewidth]{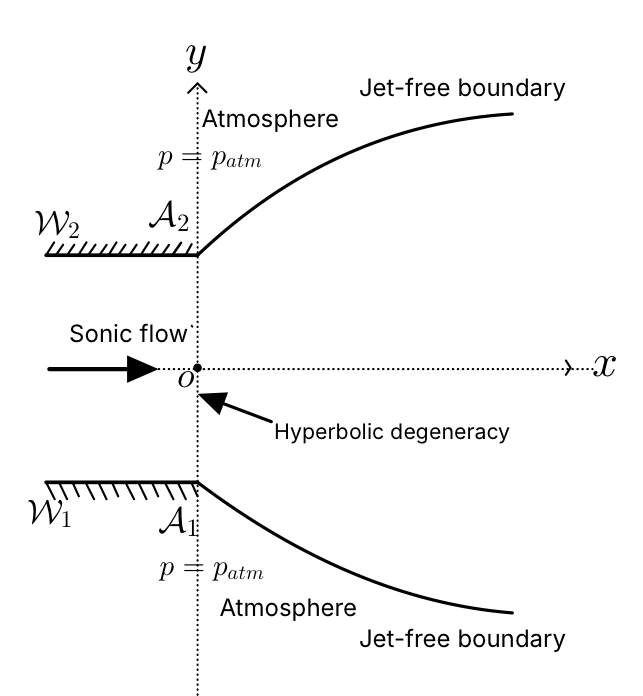}
  \caption{Jet boundary; $0<p_{atm}<p_0$} 
  \label{Jetboundary}
\end{subfigure}
\caption{Configurations of sonic--supersonic jet flows issuing from a 2-D straight nozzle: (a) expansion into a vacuum; (b) expansion into an atmosphere.}
\label{case1}
\end{figure}
\begin{equation}\label{Boundary_p}
    \begin{cases}
        (c, u, v)(x,y)=(c_0, u_0,0), & (x,y)\in \overline{\mathcal{A}_1\mathcal{A}_2};\\
        \begin{aligned}
        p(x,y)&=p_{atm},\\
        \frac{u(x,y)}{v(x,y)}&= \psi'(y),
        \end{aligned}
        & (x,y)\in \{(x,y): x=\psi(y),\ |y|>l\};\\
        \psi(\pm l)=0.
    \end{cases}
\end{equation}
We are concerned with the sonic case $u_0=c_0$. Since the flow attains the sonic state at the nozzle exit, the Euler system becomes degenerate there and \eqref{Euler}-\eqref{Boundary_p} is transformed into a degenerate hyperbolic free BVP. The analysis is further complicated by the singular behavior near the endpoints of the nozzle exit, which plays an essential role in the construction of sonic-supersonic jet flow solutions. 

To formulate the sonic-supersonic jet flows existence exhausting into an atmosphere, we first study \eqref{Euler}-\eqref{Boundary_p} for the supersonic case $u_0>c_0$. We prove the existence of local supersonic jet flow solutions near the nozzle exit and derive uniform interior $C^{0,1}$ estimates away from the nozzle end points $\mathcal{A}_1$ and $\mathcal{A}_2$. Consequently, by means of the Arzelà-Ascoli's theorem together with a classical diagonal argument, we obtain a locally Lipschitz continuous solution to the sonic-supersonic jet flow of the free BVP \eqref{Euler}-\eqref{Boundary_p}. The corresponding result is presented in the theorem below:
\begin{theorem}\label{main2}
Assume that $u_0=c_0$ and that $-\pi<\theta_{atm}<0$.
Then there exists a sufficiently small constant $\delta>0$ such that the free boundary problem
\eqref{Euler}--\eqref{Boundary_p}
admits a local sonic--supersonic jet flow solution
\[
(c, u, v)
\in
C^{0,1}\!\left(\Gamma(\delta)\right)
\cap
C^{0}\!\left(\overline{\Gamma(\delta)}\right),
\]
where
\[
\Gamma(\delta)
=
\left\{
(x,y):
x\in \bigl(\bar \psi(y),\bar \psi(y)+\delta\bigr),\;
|y|<l+\delta
\right\},
\]
with
\[
\bar \psi(y)=
\begin{cases}
\psi(y), & |y|>l,\\
0, & |y|\le l.
\end{cases}
\]
\end{theorem}

\subsection{The Euler system and basic properties}
Considering a smooth solution, we can rewrite the system \eqref{Euler} in the form
\begin{equation}\label{1}
    \begin{cases}
        \displaystyle
        \left(\frac{u}{\tau}\right)_x+\left(\frac{v}{\tau}\right)_y=0,\\[2mm]
        \displaystyle
         uu_x+vu_y+\tau p_x=0, \\
         \displaystyle
         uv_x+vv_y+\tau p_y=0.
    \end{cases}
\end{equation}
Assuming that the flow is irrotational, i.e., $u_y=v_x$, system \eqref{1} reduces to the following form
\begin{equation}\label{2}
    \begin{cases}
        (c^2-u^2)u_x-uv(u_y+v_x)+(c^2-v^2)v_y=0,\\
        u_y-v_x=0,
    \end{cases}
\end{equation}
which is supplemented by Bernoulli's law
\begin{equation}\label{Bernaoulli}
    \frac{u^2+v^2}{2}+c^2+\frac{K((\tau-b)^2-\gamma\tau(b+\gamma\tau))}{\gamma(\tau-b)^{\gamma+2}}=\text{const.},
\end{equation}
where $c^2=-\tau^2p'(\tau)$ denotes the sound speed such that
\begin{equation*}
    p'(\tau)=\frac{-K(\gamma+1)}{(\tau-b)^{\gamma+2}}+\frac{2a}{\tau^3}, \quad p''(\tau)=\frac{K(\gamma+1)(\gamma+2)}{(\tau-b)^{\gamma+3}}-\frac{6a}{\tau^4}.
\end{equation*}
It can be verified from the derivatives $p'(\tau)$ and $p''(\tau)$ that there exists a sufficiently large $\tau_1>b$ such that for all $\tau>\tau_1$,
$p'(\tau)<0 \quad \text{and} \quad p''(\tau)>0$.

The system \eqref{2} can be rewritten in the matrix form as
\begin{equation}\label{ch_matrix}
    \begin{bmatrix}
        c^2-u^2 & -uv\\
        0 & -1 
    \end{bmatrix}\begin{bmatrix}
        u\\v
    \end{bmatrix}_x+\begin{bmatrix}
        -uv & c^2-v^2\\
        -1 & 0 
    \end{bmatrix}\begin{bmatrix}
        u\\v
    \end{bmatrix}_y=0.
\end{equation}
Eigenvalues of \eqref{ch_matrix} are $\displaystyle{\Lambda_\pm= \frac{uv\pm c\sqrt{u^2+v^2-c^2}}{u^2-c^2}}$ and $\displaystyle{\Lambda_0=\frac{v}{u}}$, which implies that the system \eqref{Euler} is of mixed type: it's subsonic for $u^2+v^2<c^2$, sonic for $u^2+v^2=c^2$ and supersonic for $u^2+v^2>c^2$. When $u^2+v^2>c^2$, the system \eqref{Euler} is hyperbolic with two families of characteristic curves defined by the integral curves $\displaystyle{C_\pm: \frac{dy}{dx}=\Lambda_\pm}$ and a family of stream lines $\displaystyle{C_0:\frac{dy}{dx}=\Lambda_0}$. In this article, we restrict our attention to supersonic and sonic-supersonic flows.

We obtain the characteristic equations of \eqref{ch_matrix} as
\begin{equation}\label{Ch_equations}
    \bar\partial^\pm u+\Lambda_\mp\bar\partial^\pm v=0,   
\end{equation}
where, 
\begin{equation}\label{conditions}
   \bar\partial^+=\cos\alpha\partial_x+\sin\alpha\partial_y, \quad \bar\partial^-=\cos\beta\partial_x+\sin\beta\partial_y.
\end{equation}
Moreover, exploring \eqref{Ch_equations} gives
\begin{equation}\label{conditions1}
    \bar\partial^0=\cos\theta\partial_x+\sin\theta\partial_y, \quad \bar\partial^0=\frac{\bar\partial^++\bar\partial^-}{2\cos\omega}.
\end{equation}
\begin{figure}
    \centering
    \includegraphics[width=0.6\linewidth]{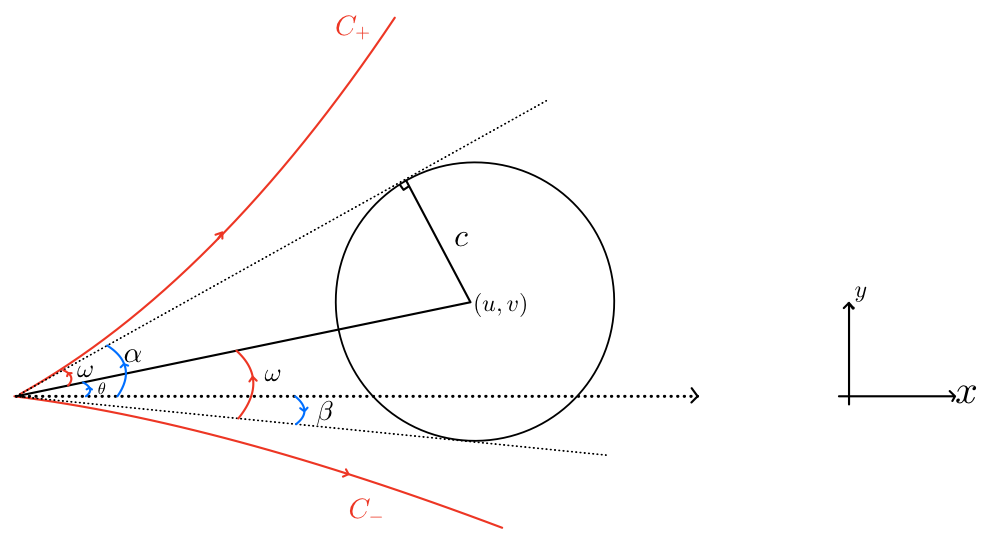}
    \caption{Characteristic curves, characteristic angles and characteristic directions.}
    \label{characteristic}
\end{figure}
The wave characteristic direction is defined as the tangent direction making an acute angle $\omega$ with the flow velocity vector $(u,v)$. A straightforward calculation shows that the $C_+$ characteristic direction forms with the direction of the flow velocity the angle $\omega$ from $(u, v)$ to $C_+$ in the counterclockwise direction and the $C_-$ characteristic direction forms with the direction of the flow velocity the angle $\omega$ from $(u, v)$ to $C_-$ in the clockwise direction; see Figure \ref{characteristic}. Furthermore, we have
\begin{equation*}
c^2=(u^2+v^2)\sin^2 \omega.
\end{equation*}
The angle $\omega$ is referred to as the Mach angle.

Following \cite{chdecomposition}, we use the concept of characteristic angle. The characteristic angle associated with the $C_+$ ($C_-$, respectively) family is defined as the counterclockwise angle measured from the positive $x$-axis to the corresponding $C_+$ ($C_-$, respectively) characteristic direction. We denote these angles by $\alpha$ and $\beta$, respectively and assume that $0\leq \alpha-\beta\leq \pi$. Furthermore, let $\theta$ denote the counterclockwise angle from the positive $x$-axis to the flow velocity vector. Then it follows that
\begin{equation}\label{5}
    \begin{split}
        \begin{cases}
           \displaystyle
            \theta =\frac{1}{2}(\alpha+\beta),\quad \omega=\frac{1}{2}(\alpha-\beta),\\[2mm]
            \displaystyle
        u=\frac{c\cos\theta}{\sin\omega},\quad v=\frac{c\sin\theta}{\sin\omega}.
        \end{cases}
    \end{split}
\end{equation}
A direct computation yields
\begin{equation*}
    (u^2+v^2)\cos\alpha\cos\beta=u^2-c^2, \quad (u^2+v^2)\sin\alpha\sin\beta=v^2-c^2.
\end{equation*}
Define
\begin{equation}\label{Bernaoulli1}
    \frac{u^2+v^2}{2}+\frac{K((\gamma+1)\tau-b)}{\gamma(\tau-b)^{\gamma+1}}-\frac{2a}{\tau}=\hat B.
\end{equation}
Exploiting \eqref{Bernaoulli1} and last two equations of \eqref{1} yields
\begin{equation*}
    u\hat B_x+v\hat B_y=0.
\end{equation*}
Also from \eqref{Bernaoulli1} we have
\begin{equation*}
    \tau_x=\frac{\tau}{c^2}(uu_x+vv_x-\hat B_x), \quad \tau_y=\frac{\tau}{c^2}(uu_y+vv_y-\hat B_y).
\end{equation*} 
Further using (\ref{conditions}) we obtain
\begin{equation}\label{derivatives}
    \begin{split}
        &\partial_x = \frac{\sqrt{u^2+v^2}}{2c\cos\omega}(\sin\alpha\bar\partial^--\sin\beta\bar\partial^+),\\
        &\partial_y = \frac{\sqrt{u^2+v^2}}{2c\cos\omega}(\cos\beta\bar\partial^+-\cos\alpha\bar\partial^-).
    \end{split}
\end{equation} 
The following notations are used throughout the article:
\begin{equation*}
    \begin{cases}
        \displaystyle
        \kappa(\tau)=-\frac{2p'(\tau)}{2p'(\tau)+\tau p''(\tau)}=\frac{2-\frac{2b}{\tau}-\frac{4a(\tau-b)^{\gamma+3}}{K(\gamma+1)\tau^4}}{\gamma+\frac{2b}{\tau}+\frac{2a(\tau-b)^{\gamma+3}}{K(\gamma+1)\tau^4}},\\[2mm]
       \displaystyle
    m(\tau)=\frac{\kappa(\tau)-1}{\kappa(\tau)+1}=\frac{2-\gamma-\frac{4b}{\tau}-\frac{2a(\tau-b)^{\gamma+3}}{K(\gamma+1)\tau^4}}{\gamma+2+\frac{6a(\tau-b)^{\gamma+3}}{K(\gamma+1)\tau^4}},\\[2mm]
    \displaystyle
    \mu^2(\tau)=\frac{1}{1+\kappa(\tau)}, \quad \mathcal{Q}(\tau)=m(\tau)-\tan^2\omega. 
    \end{cases}
\end{equation*}

\subsection{Characteristic equations in terms of characteristic angles}

In this subsection, we present the first-order characteristic decompositions for the characteristic angles. 

Exploring \eqref{5} and \eqref{conditions} we obtain
\begin{equation}\label{ch1_1}
    \begin{cases}
        \displaystyle
        \bar\partial^\pm u= \frac{\cos\theta}{\sin\omega}\bar\partial^\pm c+\frac{c}{2\sin^2\omega}(\cos\alpha\bar\partial^\pm \beta-\cos\beta\bar\partial^\pm\alpha),\\[2mm]
        \displaystyle
        \bar\partial^\pm v= \frac{\sin\theta}{\sin\omega}\bar\partial^\pm c+\frac{c}{2\sin^2\omega}(\sin\alpha\bar\partial^\pm \beta-\sin\beta\bar\partial^\pm\alpha).
    \end{cases}
\end{equation}
Inserting \eqref{ch1_1} into \eqref{Ch_equations} we compute
\begin{equation}\label{ch1_2}
    \begin{cases}
        \displaystyle
        \bar\partial^+ c=\frac{c(\bar\partial^+\alpha-\cos 2\omega\bar\partial^+\beta)}{\sin 2\omega}, \\[2mm]
        \displaystyle
        \bar\partial^-c=\frac{c(\cos 2\omega\bar\partial^-\alpha-\bar\partial^-\beta)}{\sin 2\omega}.
    \end{cases}
\end{equation}
In addition, invoking \eqref{Bernaoulli}, we obtain
\begin{equation}\label{ch1_3}
    \begin{cases}       
        u\bar\partial^+u+v\bar\partial^+v+\tau p'(\tau)\bar\partial^+\tau=0,\\
        u\bar\partial^-u+v\bar\partial^-v+\tau p'(\tau)\bar\partial^-\tau=0.
    \end{cases}
\end{equation}
Utilizing the expression of sound speed $c^2=-\tau^2 p'(\tau)$ we have
\begin{equation}\label{ch1_4}
    \begin{cases}
        \displaystyle
        \bar\partial^+\tau=-\frac{c\bar\partial^+c}{\tau(2p'(\tau)+\tau p''(\tau))},\\[2mm]
        \displaystyle
        \bar\partial^-\tau=-\frac{c\bar\partial^-c}{\tau(2p'(\tau)+\tau p''(\tau))}.
    \end{cases}
\end{equation}
Substituting \eqref{ch1_1}, \eqref{ch1_4} into \eqref{ch1_3} to get
\begin{equation}\label{ch1_5}
    \begin{split}
        \left(\frac{1}{\sin^2 \omega}+\kappa(\tau) \right)\bar\partial^+c=\frac{c\cos\omega(\bar\partial^+\alpha-\bar\partial^+\beta)}{2\sin^3\omega},\\
        \left(\frac{1}{\sin^2 \omega}+\kappa(\tau) \right)\bar\partial^-c=\frac{c\cos\omega(\bar\partial^-\alpha-\bar\partial^-\beta)}{2\sin^3\omega}.
    \end{split}
\end{equation}
Combining \eqref{ch1_2} and \eqref{ch1_5} gives
\begin{equation}\label{ch1_6}
    \begin{cases}
         \bar\partial^+\alpha=m(\tau) \cos 2\omega \bar\partial^+\beta,\\
          \bar\partial^-\beta=m(\tau)\cos 2\omega \bar\partial^-\alpha.
    \end{cases}
\end{equation}
Next, employ \eqref{ch1_5} and \eqref{ch1_6} to discover
\begin{equation}\label{ch1_7}
    \begin{cases}
        \displaystyle
        \bar\partial^+c=-c\mu^2(\tau)\cot\omega\bar\partial^+\beta,\\[2mm]
        \displaystyle
        \bar\partial^-c=c\mu^2(\tau)\cot\omega\bar\partial^-\beta,\\[2mm]
        \displaystyle
        c\bar\partial^+\alpha=\left(\frac{1-\kappa(\tau)\cos 2\omega}{2\cos^2\omega}\right)\sin 2\omega\bar\partial^+c,\\[2mm]
        \displaystyle
        c\bar\partial^-\beta=\left(\frac{\kappa(\tau)\cos 2\omega-1}{2\cos^2\omega}\right)\sin 2\omega\bar\partial^-c.
    \end{cases}
\end{equation}
Inserting \eqref{ch1_7} into \eqref{ch1_1} we obtain
\begin{equation}\label{ch1_8}
    \begin{cases}
        \bar\partial^+u=\kappa(\tau)\sin\beta\bar\partial^+c,\\
        \bar\partial^-u=-\kappa(\tau)\sin\alpha\bar\partial^-c,\\
        \bar\partial^+v=-\kappa(\tau)\cos\beta\bar\partial^+c,\\
        \bar\partial^-v=\kappa(\tau)\cos\alpha\bar\partial^-c.
    \end{cases}
\end{equation}

\subsection{Second order characteristic decompositions}

In this subsection, we derive the characteristic decomposition equations for the variable $c$. From \eqref{ch1_8}, it follows that $\nabla u$ and $\nabla v$ can be controlled through estimates on $\nabla c$. Therefore, our main objective is to obtain suitable estimates for the directional derivatives $\bar\partial^+c$ and $\bar\partial^-c$. 

\begin{lemma}
    The commutator relation between $\bar\partial^+$ and $\bar\partial^-$ takes the following form \cite{Jiequan2006}
    \begin{equation}\label{ch2_1}
        \bar\partial^-\bar\partial^+-\bar\partial^+\bar\partial^-=\frac{1}{\sin 2\omega}\left\{\left[\cos 2\omega\bar\partial^-\alpha-\bar\partial^+\beta \right]\bar\partial^++ \left[\cos 2\omega\bar\partial^+\beta-\bar\partial^-\alpha \right]\bar\partial^- \right\}.
    \end{equation}
\end{lemma}

\begin{corollary}\label{corollary1} The characteristic decomposition of the variable $c$ is 
    \begin{equation}\label{ch2_2}
    \begin{split}
        c\bar\partial^-\bar\partial^+c&=\frac{\bar\partial^+c\bar\partial^+c}{2\mu^2(\tau)\cos^2\omega} + \frac{\bar\partial^+c\bar\partial^-c}{2\cos^2\omega}\left(\frac{1}{\mu^2(\tau)}-\kappa(\tau)\sin^2 2\omega \right) + \tau\kappa'(\tau)\bar\partial^+c\bar\partial^-c,\\
        c\bar\partial^+\bar\partial^-c&=\frac{\bar\partial^-c\bar\partial^-c}{2\mu^2(\tau)\cos^2 \omega}+\frac{\bar\partial^+c\bar\partial^-c}{2\cos^2 \omega}\left(\frac{1}{\mu^2(\tau)}-\kappa(\tau)\sin^2 2\omega \right)+ \tau\kappa'(\tau)\bar\partial^+c\bar\partial^-c.
    \end{split}
\end{equation}
\begin{proof}
    Applying the commutator identity \eqref{ch2_1} to the variable $u$, it follows that
    \begin{equation}\label{ch2_3}
        \bar\partial^-\bar\partial^+u-\bar\partial^+\bar\partial^-u=\frac{1}{\sin 2\omega}\left\{\left[\cos 2\omega\bar\partial^-\alpha-\bar\partial^+\beta \right]\bar\partial^+u+ \left[\cos 2\omega\bar\partial^+\beta-\bar\partial^-\alpha \right]\bar\partial^-u \right\}.
    \end{equation}
     Exploiting \eqref{ch1_7} and \eqref{ch1_8} into \eqref{ch2_3} and using \eqref{ch2_1} for $c$ we can easily derive \eqref{ch2_2}.
\end{proof}
\end{corollary}
Utilizing \eqref{ch2_2} we can derive
\begin{equation}\label{ch2_4}
\resizebox{\linewidth}{!}{$
    \begin{cases}
        \displaystyle
        c\bar\partial^-\left(\frac{\bar\partial^+c}{c}\right)= \frac{1}{c}\left[\frac{(\bar\partial^+c)^2}{2\mu^2(\tau)\cos^2 \omega} + \frac{\bar\partial^+c\bar\partial^-c}{2\cos^2 \omega}\left(\frac{1}{\mu^2(\tau)}-\kappa(\tau)\sin^2 2\omega \right) + (\tau\kappa'(\tau)-1)\bar\partial^+c\bar\partial^-c\right],\\[3mm]
        \displaystyle
        c\bar\partial^+\left(\frac{\bar\partial^-c}{c}\right)= \frac{1}{c}\left[\frac{(\bar\partial^-c)^2}{2\mu^2(\tau)\cos^2 \omega}+\frac{\bar\partial^+c\bar\partial^-c}{2\cos^2 \omega}\left(\frac{1}{\mu^2(\tau)}-\kappa(\tau)\sin^2 2\omega \right) + (\tau\kappa'(\tau)-1)\bar\partial^+c\bar\partial^-c\right].
    \end{cases}
$}
\end{equation}

\begin{corollary} For the variable $\rho$, we have the following characteristic decomposition \cite{MR4549820}:
    \begin{equation*}
    \begin{split}
        c\bar\partial^+\bar\partial^-\rho&=\sin 2\omega \bar\partial^-\rho+\frac{\tau^4p''(\tau)}{4c \cos^2\omega}\left( \bar\partial^-\rho\bar\partial^-\rho+(f-1)\bar\partial^-\rho\bar\partial^+\rho\right),\\
        c\bar\partial^-\bar\partial^+\rho&=\sin 2\omega \bar\partial^+\rho+\frac{\tau^4p''(\tau)}{4c \cos^2\omega}\left( \bar\partial^+\rho\bar\partial^+\rho+(f-1)\bar\partial^-\rho\bar\partial^+\rho\right),
    \end{split}
\end{equation*}
where $f=2\sin^2 \omega-\frac{8p'(\tau)\cos^4A}{\tau p''(\tau)}>0$ as $p''(\tau)>0, p'(\tau)<0$.
\end{corollary}
\begin{corollary}
    The variable $c$ satisfies the following second-order characteristic decomposition
    \begin{equation}\label{ch2_6}
    \resizebox{\linewidth}{!}{$
        \begin{split}
        c\bar\partial^-\left(\frac{\cos\omega}{\bar\partial^+c}\right)&=-\left(\frac{\cos\omega}{\bar\partial^+c}\right)^2
            \left[     
                \frac{\bar\partial^+c\bar\partial^+c}{2\mu^2(\tau)\cos^3\omega} + \frac{\bar\partial^+c\bar\partial^-c}{\cos\omega}\left(\frac{1+\kappa(\tau)\cos^2 2\omega}{2\cos^2 \omega}+\tau\kappa'(\tau) \right)+ \frac{c\sin\omega\bar\partial^+c\bar\partial^-\omega}{\cos^2 \omega}
            \right],\\
           c\bar\partial^+\left(\frac{\cos\omega}{\bar\partial^-c}\right)&=-\left(\frac{\cos\omega}{\bar\partial^-c}\right)^2 \left[     
                \frac{\bar\partial^-c\bar\partial^-c}{2\mu^2(\tau)\cos^3\omega} + \frac{\bar\partial^+c\bar\partial^-c}{\cos\omega}\left(\frac{1+\kappa(\tau)\cos^2 2\omega}{2\cos^2 \omega}+\tau\kappa'(\tau) \right)+\frac{c\sin\omega\bar\partial^-c\bar\partial^+\omega}{\cos^2 \omega}
            \right].               
        \end{split}
    $}    
    \end{equation}
\end{corollary}

\section{Supersonic and sonic-supersonic jet flows in the vacuum case}\label{section_3}

In this section, we study the global existence of supersonic and sonic-supersonic flow solutions to the problem \eqref{Euler}, \eqref{Boundary_v} for $u_0\geq c_0$. Besides the existence, we also establish some uniform interior estimates for the supersonic flow solutions with respect to $u_0\geq c_0$.

\subsection{Centered simple wave}

\begin{theorem}\label{principalpart}
    Let $(c,u,v)(x,y)=(\hat c,\hat u,\hat v)(\alpha)$ be a $C_+$ ($C_-$, respectively) centered simple wave solution with $y=x\tan\alpha,\;\alpha<\alpha_0$. Then the principal variables $(\hat c,\hat u,\hat v)(\alpha)$ satisfy the following system \cite{Courant1948}:
\begin{equation}\label{principalpart1}
    \begin{cases}
        \displaystyle
        \cos\alpha\frac{d\hat u}{d\alpha}+\sin\alpha\frac{d\hat v}{d\alpha}=0,\\[2mm]
        \displaystyle
         \frac{\hat u^2+\hat v^2}{2}+\frac{K((\gamma+1)\hat \tau-b)}{\gamma(\hat \tau-b)^{\gamma+1}}-\frac{2a}{\hat \tau}=\text{const.},\\[2mm]
         \hat u\sin\alpha-\hat v\cos\alpha=\hat c.
    \end{cases}
\end{equation}
\end{theorem}
\begin{theorem}\label{simplewaveth}
Assume that the incoming flow is a uniform supersonic flow given by $(c, u, v)=(c_0, u_0,0)$ and $\tau>b$. Then there exist centered simple wave $S_+$ ($S_-$, respectively) consisting of straight $C_+$ ($C_-$, respectively) characteristic lines emanating from the nozzle endpoint $\mathcal{A}_1$ ($\mathcal{A}_2$, respectively), which connect the uniform flow state $(c, u, v)=(c_0,u_0,0)$ to the vacuum state. Moreover, the following relation holds:
\begin{equation}\label{Simple_1}
   \begin{cases}
       \hat u(\alpha)&=c\sin\alpha+\sqrt{\mathcal{F(\hat\tau)}+\mathcal{F}_1}\cos\alpha,\\
       \hat v(\alpha)&=-c\cos\alpha+\sqrt{\mathcal{F(\hat\tau)}+\mathcal{F}_1}\sin\alpha,\\
       \hat c(\alpha)&=\int_{\alpha_0}^\alpha \mu^2(\hat\tau)\sqrt{\mathcal{F(\hat\tau)}+\mathcal{F}_1}\;d\alpha+c_0,
   \end{cases}
\end{equation}
where 
\begin{equation}\label{Simple_2}
    \begin{split}
        \mathcal{F(\hat\tau)}&=-\frac{K}{\gamma(\hat\tau-b)^{\gamma+2}}\left((\gamma+1)(\gamma+2)\hat\tau^2-2(\gamma+2)b\hat\tau+2b^2 \right)+\frac{6a}{\hat\tau},\\
    \mathcal{F}_1&=u_0^2-c_0^2+\frac{K}{\gamma(\hat\tau_0-b)^{\gamma+2}}\left((\gamma+1)(\gamma+2)\hat\tau_0^2-2(\gamma+2)b\hat\tau_0+2b^2 \right)+\frac{6a}{\hat\tau_0}.
    \end{split}
\end{equation}
Further, we define
\begin{equation}\label{Simple_18}
    \alpha_v:=\int_{\tau_0}^\infty \frac{\sqrt{-p'(\tau)}\sqrt{u^2+v^2+\tau^2p'(\tau)}}{u^2+v^2}d\tau, 
\end{equation}
where $u^2+v^2=\mathcal{F}(\hat\tau)+\mathcal{F}_1+c^2.$
\end{theorem}
\begin{proof}
Combining \eqref{ch1_4} and \eqref{ch1_7}, we obtain 
\begin{equation*}
         \begin{split}
             \bar\partial^-\alpha &=-\frac{\sqrt{u_0^2-2\int_{\tau_0}^\tau hp'(h)dh+\tau^2p'(\tau)}}{2\tau\sqrt{-p'(\tau)}}\bar\partial^-\tau\\
             &=-\frac{\sqrt{u_0^2+\tau_0^2p'(\tau_0)+\int_{\tau_0}^\tau h^2p''(h)dh}}{2\tau\sqrt{-p'(\tau)}}\bar\partial^-\tau.
         \end{split}
     \end{equation*}
Integrating the above identity along the $C_-$ characteristic yields
\begin{equation}\label{Simple_4}
    \hat\alpha(\tau):=\alpha(x,y)=\alpha_0-\int_{\tau_0}^\tau \frac{\sqrt{u_0^2+\tau_0^2p'(\tau_0)+\int_{s_0}^s h^2p''(h)dh}}{2s\sqrt{-p'(s)}}ds.
\end{equation}
Differentiating \eqref{Simple_4} with respect to $\tau$ gives 
\begin{equation}\label{Simple_5}
    \frac{d\alpha}{d\tau}=-\frac{\sqrt{u_0^2+\tau_0^2p'(\tau_0)+\int_{\tau_0}^\tau h^2p''(h)dh}}{2\tau\sqrt{-p'(\tau)}}<0.
\end{equation}
From \eqref{Simple_4}-\eqref{Simple_5}, together with the relation $c^2=-\tau^2p'(\tau)$, it follows that $\tau(x,y)=\hat{\tau}(\alpha)$ and $(c, u, v)(x,y)=(\hat{c}, \hat{u}, \hat{v})(\alpha)$.
Next, applying the Bernoulli relation together with the incoming flow $(c_0,u_0,0)$, we obtain
\begin{equation}\label{Simple_6}
\frac{\hat u^2+\hat v^2}{2}+\hat c^2+\frac{K((\hat\tau-b)^2-\gamma\hat\tau(b+\gamma\hat\tau))}{\gamma(\hat\tau-b)^{\gamma+2}}
    =\frac{u_0^2}{2}+c_0^2+\frac{K((\tau_0-b)^2-\gamma\tau_0(b+\gamma\tau_0))}{\gamma(\tau_0-b)^{\gamma+2}}.
\end{equation}
Furthermore, using \eqref{Simple_6} and recalling $u=\sqrt{u^2+v^2}\cos\theta,\; v=\sqrt{u^2+v^2}\sin\theta$, we deduce 
\begin{equation*}
     \theta(\tau):=\theta(x,y)= -\int_{\tau_0}^{\hat\tau} \frac{\sqrt{-p'(\tau)}\sqrt{u^2+v^2+\hat\tau^2p'(\tau)}}{u^2+v^2}d\tau+\theta(\tau_0).
\end{equation*} 
To facilitate the subsequent analysis, we decompose $(\hat u,\hat v)(\alpha)$ along the orthogonal directions $(\cos\alpha,\sin\alpha)$ and $(\sin\alpha,-\cos\alpha)$ as 
\begin{equation}\label{Simple_8}
    \hat g=\hat u \cos\alpha+\hat v\sin\alpha, \; \hat c=\hat u \sin\alpha-\hat v\cos\alpha.
\end{equation}
Accordingly, 
\begin{equation}\label{Simple_9}
    \hat u=\hat g \cos\alpha+\hat c\sin\alpha, \; \hat v=\hat g \sin\alpha-\hat c\cos\alpha,
\end{equation}
which immediately yields 
\begin{equation*}
    \hat g^2+\hat c^2=\hat u^2+\hat v^2.
\end{equation*}
Differentiating \eqref{Simple_9} with respect to $\alpha$, we find
\begin{equation}\label{Simple_11}
    \begin{cases}
        \displaystyle
        \frac{d\hat u}{d\alpha}=\frac{d\hat g}{d\alpha}\cos\alpha-\hat g\sin\alpha+\frac{d\hat c}{d\alpha}\sin\alpha+\hat c\cos\alpha,\\[2mm]
        \displaystyle
        \frac{d\hat v}{d\alpha}=\frac{d\hat g}{d\alpha}\sin\alpha+\hat g\cos\alpha-\frac{d\hat c}{d\alpha}\cos\alpha+\hat c\sin\alpha.
    \end{cases}
\end{equation}
Making use of \eqref{Simple_11} and \eqref{principalpart1}, we obtain
\begin{equation}\label{Simple_12}
    \frac{d\hat g}{d\alpha}=-\hat c.
\end{equation}
Exploring \eqref{Simple_6}, \eqref{Simple_8} and \eqref{Simple_11} we have
\begin{equation}\label{Simple_13}
    \hat g(\alpha)=(1+\kappa(\tau))\frac{d\hat c}{d\alpha}.
\end{equation}
Combining \eqref{Simple_12} with \eqref{Simple_13}, we derive
\begin{equation}\label{Simple_14}
    \frac{d^2\hat c}{d\alpha^2}-\frac{\kappa'(\tau)\kappa(\tau)\mu^2(\tau)}{\sqrt{-p'(\tau)}}\left(\frac{d\hat c}{d\alpha}\right)^2+\hat c=0.
\end{equation}
Finally, imposing the initial conditions $\hat c(\alpha_0)=c_0$ and $\hat g(\alpha_0)=\sqrt{u_0^2-c_0^2}$ together with \eqref{Simple_14}, we obtain
\begin{equation}\label{Simple_15}
   \hat c:= c(x,y)=c_0+\int_{\alpha_0}^\alpha \mu^2(\hat\tau)\sqrt{\mathcal{F}(\hat\tau)+\mathcal{F}_1}d\alpha,
\end{equation}
where $\mathcal{F}(\tau)$ and $\mathcal{F}_1$ are given in \eqref{Simple_2}.

It then follows from \eqref{Simple_13} and \eqref{Simple_14} that
\begin{equation}\label{Simple_16}
    \hat g=(1+\kappa(\hat\tau))\hat c_\alpha=\sqrt{\mathcal{F}(\hat\tau)+\mathcal{F}_1},
\end{equation}
where $\mathcal F(\tau)$ and $\mathcal F_1$ are defined in \eqref{Simple_2}.

Substituting \eqref{Simple_15} and \eqref{Simple_16} into \eqref{Simple_9}, we conclude that 
\begin{equation*}
    \begin{split}
         u(\alpha)&=\sqrt{\mathcal{F(\tau)}+\mathcal{F}_1}\cos\alpha+c\sin\alpha,\\
         v(\alpha)&=\sqrt{\mathcal{F(\tau)}+\mathcal{F}_1}\sin\alpha-c\cos\alpha,
    \end{split}
\end{equation*}
where $\mathcal{F}(\tau)$ and $\mathcal{F}_1$ are given in \eqref{Simple_2}.

Consider $\displaystyle{\alpha_v=\lim_{\hat\tau\to \infty}\theta}$ then
\begin{equation*}
    \alpha_v=\int_{\tau_0}^{\infty} \frac{\sqrt{-p'(\tau)}\sqrt{u^2+v^2+\tau^2p'(\tau)}}{u^2+v^2}d\tau.
\end{equation*}
A direct computation shows that the integral is absolutely convergent. Consequently, $\alpha_v$ is a bounded quantity. This completes the proof.
\end{proof}


\subsection{Supersonic jet flow expanding into the vacuum}\label{vacuum_subsection1}
\begin{figure}
    \centering
    \includegraphics[width=0.5\linewidth]{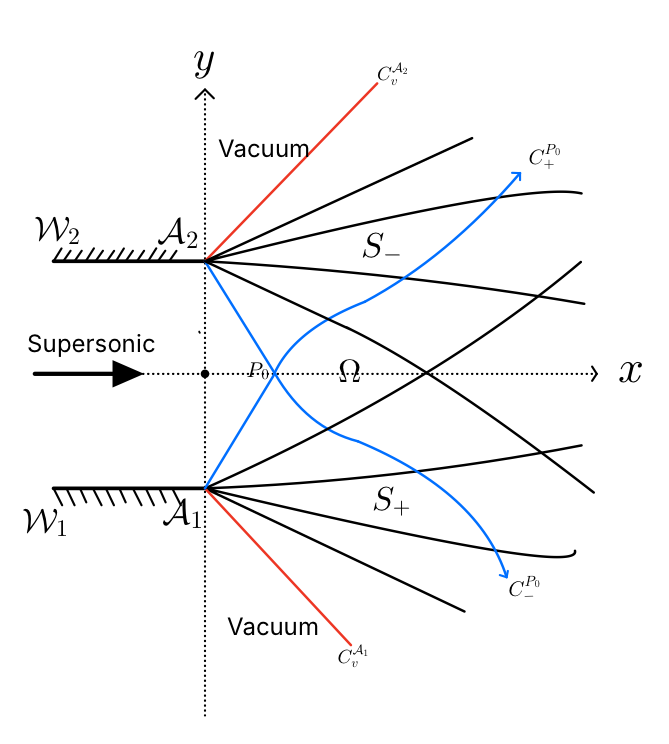}
    \caption {Configuration of a jet expanding into a vacuum from a straight two-dimensional nozzle.}
    \label{Simplewave}
\end{figure}
First, we focus on the free boundary problem \eqref{Euler}--\eqref{Boundary_v} in the supersonic regime $u_0>c_0$. As shown by Courant and Friedrichs \cite{Courant1948}, the jet flow consists of two simple wave regions issuing from the nozzle endpoints $\mathcal{A}_1$ and $\mathcal{A}_2$; see Figure \ref{Simplewave}. Since the characteristics of a centered simple wave are straight lines, every point inside the upper and lower simple wave regions can be uniquely identified by its characteristic angle. Consequently, the characteristic angle is given by
\begin{equation*}
    \alpha=\arctan\left(\frac{y+l}{x}\right)
\quad\text{and}\quad
\alpha=\arctan\left(\frac{l-y}{x}\right),
\end{equation*}
in the upper and lower simple wave regions, respectively.

Exploiting these geometric relations into \eqref{simplewaveth}, we obtain
\begin{equation*}
    \begin{split}
        (c, u, v)=(c_1,u_1,v_1)(x,y)&=(\hat c, \hat u, \hat v)(\alpha), \quad \alpha_v<\arctan\left(\frac{l+y}{x}\right)<\alpha_0,\\
        (c, u, v)=(c_2,u_2,v_2)(x,y)&=(\hat c, \hat u, -\hat v)(\alpha), \quad \alpha_v<\arctan\left(\frac{l-y}{x}\right)<\alpha_0,
    \end{split}
\end{equation*}
where $\displaystyle{\alpha_0=\arcsin\left(\frac{c_0}{u_0}\right)}$ and $\alpha_v$ is given in \eqref{Simple_18}.

The two centered simple waves intersect at the point
\[
P_0=(l\cot\alpha_0,0).
\]
To describe this interaction, we introduce the forward cross-characteristic curves $C_-^{P_0}$ and $C_+^{P_0}$ passing through $P_0$ in the simple wave regions $S_+$ and $S_-$, respectively. These curves serve as the initial curves for the subsequent interaction analysis.

We first determine the geometric properties of $C_\pm^{P_0}$ together with the corresponding boundary data. Parameterizing $C_-^{P_0}$ by the characteristic angle $\alpha$, we write
\begin{equation*}
(x,y)=\bigl(r(\alpha)\cos\alpha,\,-l+r(\alpha)\sin\alpha\bigr),
\qquad
\alpha\in(\alpha_v,\alpha_0).
\end{equation*}
Using the simple wave solution \eqref{Simple_1}, the Mach angle along $C_-^{P_0}$ is given by
\begin{equation*}
    \omega=\hat \omega(\alpha)=
    \arcsin\!\left(
    \frac{\hat c(\alpha)}
    {\sqrt{\hat u^2(\alpha)+\hat v^2(\alpha)}}
    \right),
    \qquad \text{on } C_-^{P_0}.
\end{equation*}
Since $\beta=\alpha-2\hat \omega(\alpha)$ and the tangent vector to $C_-^{P_0}$ satisfies
\[
\sin\beta\,\frac{dx}{d\alpha}
-
\cos\beta\,\frac{dy}{d\alpha}
=0,
\]
it follows that
\begin{equation*}
    \frac{dr(\alpha)}{d\alpha}
    =
    -r(\alpha)\cot\bigl(2\hat \omega(\alpha)\bigr).
\end{equation*}
Consequently,
\begin{equation}\label{3.2_6}
    r(\alpha)>0
    \quad\text{for}\quad
    \alpha\in(\alpha_v,\alpha_0]
    \qquad\text{and}\qquad
    \lim_{\alpha\to\alpha_v}r(\alpha)=\infty.
\end{equation}
Equation \eqref{3.2_6} implies that the cross-characteristic curve $C_-^{P_0}$ remains entirely inside the upper simple wave region and never intersects the vacuum boundary
\[
C_v^{\mathcal{A}_1}=\left\{(x,y): (x,y)=(r\cos\alpha_v,\; r\sin\alpha_v-l),\; r>0\right\}.
\]
Hence, the solution in the region enclosed by $\overline{{\mathcal{A}_1}P_0}$, $C_-^{P_0}$ and $C_v^{\mathcal{A}_1}$ is completely determined by the centered simple wave constructed in Theorem \ref{simplewaveth}. Throughout the sequel, this region will be denoted by $S_+$.

By symmetry, the same conclusion holds for the lower simple wave region. In particular, the cross-characteristic curve $C_+^{P_0}$ extends throughout the lower simple wave region without intersecting the vacuum boundary
\[
C_v^{\mathcal{A}_2}=\left\{(x,y) : (x,y)=(r\cos\alpha_v,\; -r\sin\alpha_v+l),\; r>0\right\},
\]
and the flow in the region bounded by $\overline{{\mathcal{A}_2}P_0}$, $C_+^{P_0}$ and $C_v^{\mathcal{A}_2}$ is described by the centered simple wave. For convenience, we continue to denote this region by $S_-$.

To determine the boundary data along the cross-characteristic curve $C_-^{P_0}$, we evaluate the characteristic derivative $\bar\partial^-c$. Since the characteristic angle satisfies
\[
\alpha=\arctan\left(\frac{y+l}{x}\right)
\quad\text{on } \quad C_-^{P_0},
\]
it follows that
\begin{equation}\label{3.2_7}
    \bar\partial^-c=-\frac{2c\cos\omega\mu^2(\tau)\sqrt{\mathcal{F}(\tau)+\mathcal{F}_1}}{\sqrt{u^2+v^2}\sqrt{x^2+(y+l)^2}}<0 \quad \text{on} \quad C^{P_0}_- .
\end{equation}
By symmetry,
\begin{equation}\label{3.2_8}
    \bar\partial^+c<0 \qquad \text{on} \qquad C_+^{P_0}.
\end{equation}
Having determined the simple wave solutions in $S_-$ and $S_+$, it remains to analyze the interaction of the two wave families beyond the cross-characteristics $C_+^{P_0}$ and $C_-^{P_0}$. This interaction is described by the Euler system \eqref{Euler} subject to the boundary conditions
\begin{equation}\label{3.2_9}
(c, u, v)|_{C_-^{P_0}} =(c_1,u_1,v_1)(x,y),\quad
(c, u, v)|_{C_+^{P_0}}=(c_2,u_2,v_2)(x,y).
\end{equation}
\begin{lemma}\label{3.2_l1}
There exists a sufficiently small constant $\varepsilon>0$ such that the Goursat problem \eqref{Euler}--\eqref{3.2_9} admits a solution in the domain $\Omega_\varepsilon$, enclosed by the cross-characteristic curves $C_\pm^{P_0}$ and the level curve $c(x,y)=c_0-\varepsilon$. Moreover, the solution satisfies $\displaystyle{\bar{\partial}^\pm c<0,}$ throughout $\Omega_\varepsilon$.
\end{lemma}
\begin{proof}
To facilitate the proof, fix a sufficiently small positive constant $\eta>0$, the vertical line
\[
x=\eta+l\cot\alpha_0,
\]
intersects $C_-^{P_0}$ and $C_+^{P_0}$ at the points $P_1$ and $P_2$, respectively; see Figure \ref{levelcurve}. Consequently, the boundary value problem \eqref{Euler}-\eqref{3.2_9} reduces to a Goursat problem posed on the characteristics $\widehat{P_0P_1}$ and $\widehat{P_0P_2}$.

By the classical theory of the method of characteristics (see, \cite{liboundaryvalue}), there is a small positive constant $\eta>0$ for which the Goursat problem has a unique classical solution in the domain enclosed by the characteristic segments $\widehat{P_0P_1}$, $\widehat{P_0P_2}$ and the line segment $\overline{P_1P_2}$.

Furthermore, combining the characteristic decompositions established in Corollary \ref{corollary1} with \eqref{3.2_7}--\eqref{3.2_8}, we conclude that
\begin{equation*}
    \bar\partial^\pm c<0 \quad\text{in} \quad \Omega_\varepsilon.
\end{equation*}
Define
\[
\varepsilon
=
c_0-
\sup_{\overline{P_1P_2}}c.
\]
Since $c<c_0$ inside the local interaction region, we have $\varepsilon>0$. Therefore, the Goursat problem admits a classical solution in the domain $\Omega_\varepsilon$, bounded by the characteristic curves $C_\pm^{P_0}$ and the level curve $c=c_0-\varepsilon$. This completes the proof.
\end{proof}
\begin{lemma}\label{3.2_l2}
Assume that there exists a classical solution to the Goursat problem \eqref{Euler}--\eqref{3.2_9} in the domain $\Omega_\varepsilon$ for some $\varepsilon\in(0,c_0)$. Then the solution therefore satisfies
\begin{equation}\label{3.2_10}
    \mathcal{H}
    <
    \frac{\bar\partial^+c}{c}
    <0,
    \qquad
    \mathcal{H}
    <
    \frac{\bar\partial^-c}{c}
    <0
    \quad\text{in }\Omega_\varepsilon,
\end{equation}
where
\[
\mathcal{H}
=
\inf_{C_-^{P_0}}
\frac{\bar\partial^-c}{c}.
\]
\end{lemma}
\begin{proof}
By Lemma \ref{3.2_l1}, the characteristic derivatives satisfy
\[
\bar{\partial}^{\pm}c<0
\quad\text{in}\quad\Omega_\varepsilon.
\]
Therefore, it remains to establish the lower bound
\[
\mathcal{H}<\frac{\bar{\partial}^{\pm}c}{c}.
\]
We first observe that, in view of \eqref{3.2_7}, \eqref{3.2_8} and the Bernoulli relation \eqref{Bernaoulli},
\begin{equation*}
\frac{ u^2+ v^2}{2}+ c^2+\frac{K((\tau-b)^2-\gamma\tau(b+\gamma\tau))}{\gamma(\tau-b)^{\gamma+2}}
     =\frac{u_0^2}{2}+c_0^2+\frac{K((\tau_0-b)^2-\gamma\tau_0(b+\gamma\tau_0))}{\gamma(\tau_0-b)^{\gamma+2}}.
\end{equation*}
Consequently, the quantity $\mathcal{H}$ is well defined and finite. Moreover, from the boundary data prescribed on $C_+^{P_0}$, we have
\[
\frac{\bar{\partial}^{+}c}{c}>\mathcal{H}
\quad\text{on} \quad C_+^{P_0}.
\]
We next show that
\[
\frac{\bar\partial^+c}{c}>\mathcal{H}
\quad\text{on}\quad C_-^{P_0}.
\]
Assume that $\frac{\bar\partial^+c}{c}\leq \mathcal{H}$ on $C_-^{P_0}$. Then there exists a first point $H\in C_-^{P_0}$ such that
\[
\frac{\bar\partial^+c}{c}>\mathcal{H}
\quad\text{on }\widehat{PH},
\quad\text{while}\quad
\left(\frac{\bar\partial^+c}{c}\right)(H)=\mathcal{H}.
\]
By the choice of $H$, it follows that
\[
\bar\partial^-\!\left(\frac{\bar\partial^+c}{c}\right)(H)\leq0.
\]
On the other hand, using the first equation of \eqref{ch2_4} along with \eqref{ch2_2}, we obtain
\begin{equation}\label{3.2_12}
\resizebox{\linewidth}{!}{$
    \begin{split}
        c\bar\partial^-\left(\frac{\bar\partial^+c}{c}\right)&=\frac{1}{c}\left(\frac{\bar\partial^+c\bar\partial^+c}{2\mu^2(\tau)\cos^2 \omega} + \frac{\bar\partial^+c\bar\partial^-c}{2\cos^2 \omega}\left(\frac{1}{\mu^2(\tau)}-\kappa(\tau)\sin^2 2\omega \right) + \tau\kappa'(\tau)\bar\partial^+c\bar\partial^-c-\bar\partial^+c\bar\partial^-c \right)\\
       &>\frac{\bar\partial^+c}{c}\left(\frac{1+\kappa(\tau)}{2\cos^2 \omega}\bar\partial^+c-\bar\partial^-c \right)\\
       &>\frac{\bar\partial^+c\bar\partial^-c}{2c\cos^2 \omega}\left(\kappa(\tau)+\cos 2\omega\right)>0 \quad \text{at} \quad H,
    \end{split}
$}    
\end{equation}
which contradicts
\[
\bar\partial^-\!\left(\frac{\bar\partial^+c}{c}\right)(H)\leq0.
\]
Hence,
\[
\frac{\bar\partial^+c}{c}>\mathcal{H}
\qquad\text{on }C_-^{P_0}.
\]
By the same argument, we also obtain
\[
\frac{\bar\partial^-c}{c}>\mathcal{H}
\qquad\text{on }C_+^{P_0}.
\]
Let $Q$ be an arbitrary point in $\Omega_\varepsilon$. Since the solution is constructed by the method of characteristics, the backward $C_+$ ($C_-$, respectively) characteristic passing through $Q$ remains inside $\Omega_\varepsilon$ until it intersects the cross-characteristic curve $C_-^{P_0}$ ($C_+^{P_0}$, respectively) at a point $Q_+$ ($Q_-$, respectively); see Figure \ref{levelcurve}. Denote by $\Pi_Q$ the characteristic quadrilateral enclosed by the characteristic arcs $\widehat{P_0Q_-}$, $\widehat{P_0Q_+}$, $\widehat{Q_-Q}$ and $\widehat{Q_+Q}$.

We shall prove that the estimate \eqref{3.2_10} propagates throughout $\Pi_Q$. Assume that \eqref{3.2_10} holds in $\Pi_Q\setminus\{Q\}$. Suppose, for contradiction, that
\[
\left(\frac{\bar\partial^+c}{c}\right)(Q)=\mathcal{H},
\qquad
\left(\frac{\bar\partial^-c}{c}\right)(Q)\geq\mathcal{H}.
\]
Then, by the assumption,
\[
\bar\partial^-\left(\frac{\bar\partial^+c}{c}\right)(Q)\leq0.
\]
However, applying the same computation as in \eqref{3.2_12} yields
\[
\bar\partial^-\left(\frac{\bar\partial^+c}{c}\right)(Q)>0,
\]
which is a contradiction. Hence,
\[
\left(\frac{\bar\partial^+c}{c}\right)(Q)>\mathcal{H}.
\]
An analogous argument shows that
\[
\left(\frac{\bar\partial^-c}{c}\right)(Q)>\mathcal{H}.
\]
Since $Q$ is arbitrary, the desired estimate follows by the method of continuity. This completes the proof.
\end{proof}
\begin{lemma}\label{3.2_l3}
The Goursat problem \eqref{Euler}--\eqref{3.2_9} admits a unique global classical solution in the domain
\[
\Omega=\bigcup_{\varepsilon\in(0,c_0)}\Omega_\varepsilon.
\]
Furthermore, $\Omega$ is an infinite sectorial domain enclosed by the characteristic curves $C_\pm^{P_0}$.
\end{lemma}
\begin{figure}
    \centering
    \includegraphics[width=0.7\linewidth]{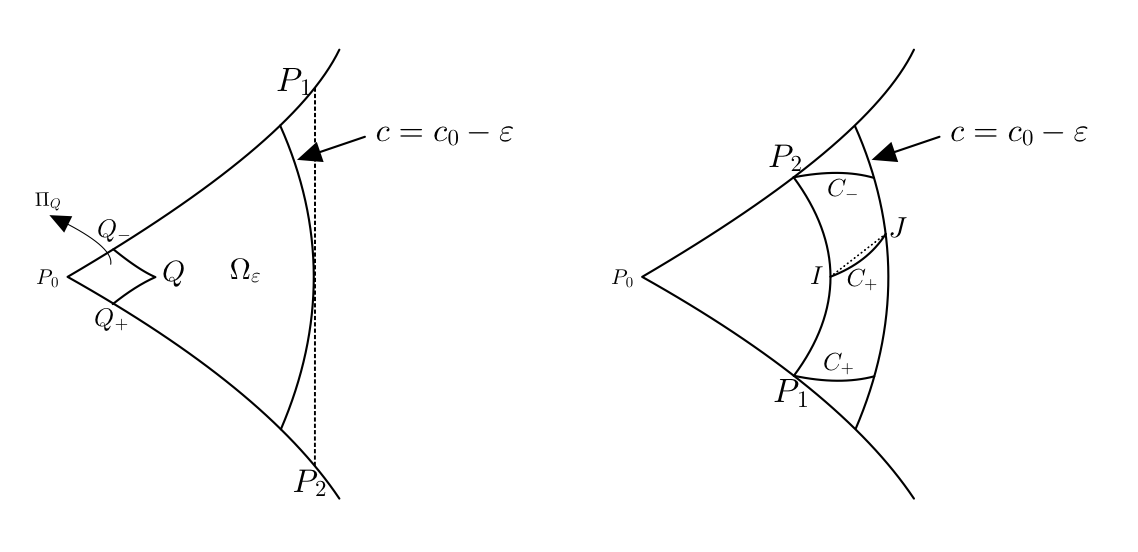}
    \caption {Illustration of the local solution and the corresponding level curve of $c$.}
    \label{levelcurve}
\end{figure}
\begin{proof}
Suppose the Goursat problem \eqref{Euler}--\eqref{3.2_9} possesses a classical solution in the domain $\Omega_\varepsilon$ for some $\varepsilon\in(0,c_0)$. By Lemma \ref{3.2_l2} together with \eqref{ch1_8}, the resulting classical solution fulfills
\begin{equation*}
    \begin{split}
       c_0-\varepsilon<c<c_0,\\
        u_0<\sqrt{u^2+v^2}<\sqrt{u_0^2+\frac{2K((\gamma+1)\tau_0-b)}{\gamma(\tau_0-b)^{\gamma+1}}-\frac{4a}{\tau_0}},\\
        |\nabla c|<-\frac{\mathcal H}{\cos\omega_0}\sqrt{u_0^2+\frac{2K((\gamma+1)\tau_0-b)}{\gamma(\tau_0-b)^{\gamma+1}}-\frac{4a}{\tau_0}},\\
        |\nabla u|<-\frac{\kappa(\tau_0)\mathcal H}{\cos\omega_0}\sqrt{u_0^2+\frac{2K((\gamma+1)\tau_0-b)}{\gamma(\tau_0-b)^{\gamma+1}}-\frac{4a}{\tau_0}},\\
        |\nabla v|<-\frac{\kappa(\tau_0)\mathcal H}{\cos\omega_0}\sqrt{u_0^2+\frac{2K((\gamma+1)\tau_0-b)}{\gamma(\tau_0-b)^{\gamma+1}}-\frac{4a}{\tau_0}},
    \end{split}
\end{equation*}
throughout $\Omega_\varepsilon$ for some $\varepsilon\in(0,c_0)$.

Moreover, Lemma \ref{3.2_l2} implies that
\[
\bar{\partial}^{\pm}c<0
\qquad\text{in }\Omega_\varepsilon,
\]
which ensures that every level curve of $c$ is non-characteristic. Consequently, the local continuation theorem for Goursat problems (cf. \cite{MR1291392}) can be applied to extend the classical solution beyond the level curve $c=c_0-\varepsilon$. Hence, there exists $\varepsilon_1>\varepsilon$ such that the solution extends uniquely to the larger domain $\Omega_{\varepsilon_1}$.

Repeating the above continuation argument and invoking the method of continuity, we conclude that the Goursat problem \eqref{Euler}--\eqref{3.2_9} admits a unique a classical solution in
\[
\Omega=\bigcup_{\varepsilon\in(0,c_0)}\Omega_\varepsilon,
\]
which is a sector-shaped domain enclosed by $C_\pm^{P_0}$; see Figure \ref{levelcurve}. 

Next, we establish
\begin{equation}\label{3.2_14}
   \operatorname{dist}(P_0,l_\varepsilon)\to+\infty
    \qquad\text{as}\qquad
    \varepsilon\to c_0,
\end{equation}
where $l_\varepsilon$ denotes the level curve $c(x,y)=c_0-\varepsilon$.

For any $\epsilon>0$, let $\alpha_\epsilon$ be determined by $\hat c(\alpha_\epsilon)=\epsilon$. Then there exists a sufficiently small $\epsilon>0$ such that
\begin{equation}\label{3.2_15}
\begin{split}
&\alpha_\epsilon+(\kappa(\tau_0)-1)\frac{\epsilon}{c_0}<0,\\
&\tan^2\omega=\frac{c^2}{(u^2+v^2)\cos^2 \omega}<\frac{c^2}{c_0^2\cos^2 \omega_0}<m(\tau_0),\quad \text{for } c<\epsilon,\\
&(\kappa(\tau_0)-1)\frac{\epsilon}{c_0}<\frac{\pi}{4}.
\end{split}
\end{equation}
Let $P_1$ and $P_2$ denote the points on the cross-characteristic curves $C_\pm^{P_0}$ satisfying
\[
c(P_1)=\epsilon, \quad c(P_2)=\epsilon.
\]
Furthermore, let $C_+^{P_1}$ ($C_-^{P_2}$, respectively) be the $C_+$ ($C_-$, respectively) characteristic curve issuing from $P_1$ ($P_2$, respectively). Then, in view of equation \eqref{ch1_7} together with the second inequality in \eqref{3.2_15}, we obtain
\begin{equation}\label{3.2_16}
    0<\bar\partial^+\alpha<-(\kappa(\tau_0)-1)\frac{\bar\partial^+c}{c_0},
    \qquad \text{along } C_+^{P_1}.
\end{equation}
Integrating \eqref{3.2_16} along the characteristic curve $C_+^{P_1}$ and using the first inequality in \eqref{3.2_15}, we arrive at
\begin{equation*}
    \alpha(P_1)=\alpha_\epsilon
    <\alpha<
    \alpha_\epsilon+(\kappa(\tau_0)-1)\frac{\epsilon}{c_0}
    <0,
    \qquad \text{on } C_+^{P_1}.
\end{equation*}
By symmetry, the corresponding estimate $\beta>0$ holds along $C_-^{P_2}$. Consequently, the characteristic curves $C_+^{P_1}$ and $C_-^{P_2}$ remain disjoint.

For any $\varepsilon>c_0-\epsilon$, let $J$ be an arbitrary point on the level curve $l_\varepsilon$. Since the characteristic curves $C_+^{P_1}$ and $C_-^{P_2}$ remain disjoint, there exists a point $I\in l_{c_0-\epsilon}$ such that $J$ lies on either the $C_+^I$ or the $C_-^I$ characteristic issuing from $I$. Without loss of generality, we assume that $J\in C_+^I$, where $C_+^I$ denotes the $C_+$ characteristic curve passing through $I$. By Lemma \ref{3.2_l2}, we have
\begin{equation*}
    \widehat{|IJ|} >-\frac{1}{\mathcal{H}}\ln\left(\frac{\varepsilon}{c_0-\epsilon}\right),
\end{equation*}
where $\widehat{|IJ|}$ denotes the arc length of the characteristic segment $\widehat{IJ}$. Moreover, arguing as in the derivation of \eqref{3.2_16}, the characteristic curve $\widehat{IJ}$ is convex and satisfies
\begin{equation}\label{3.2_19}
    0<\alpha(J)-\alpha(I)
    <
    (\kappa(\tau_0)-1)\frac{\epsilon}{c_0}.
\end{equation}
Combining \eqref{3.2_19} with the third inequality in \eqref{3.2_15}, we conclude that
\begin{equation*}
    \operatorname{dist}(I,J)>
    -\frac{1}{\sqrt{2}\mathcal{H}}
    \ln\left(\frac{\varepsilon}{c_0-\epsilon}\right).
\end{equation*}
Consequently,
\[
\operatorname{dist}(P_0,l_\varepsilon)\rightarrow+\infty
\qquad\text{as}\qquad
\varepsilon\rightarrow c_0,
\]
which establishes \eqref{3.2_14} and completes the proof.
\end{proof}
Combining the above results we conclude that, for every fixed $u_0>c_0$, there exists a global piecewise smooth supersonic jet flow expanding from the nozzle into a vacuum. The proof of Theorem \ref{main1} in the supersonic regime is therefore completed.

\subsection{A priori estimates}\label{vacuum_subsection2}

In this subsection, we derive interior $C^{0,1}$ estimates for the supersonic jet flow solutions constructed in the previous subsection \ref{vacuum_subsection1}. These estimates play an essential role in the limiting procedure used to establish the existence of sonic-supersonic jet flow solutions. To overcome the singular behavior arising near the vacuum boundary, we employ the approach introduced in \cite{MR4275747}. We begin by defining the following distance functions:
\begin{equation*}
    R_{\mathcal{A}_1}(x,y)=\operatorname{dist}\big((x,y), \mathcal{A}_1\big),
    \qquad
    R_{\mathcal{A}_2}(x,y)=\operatorname{dist}\big((x,y),\mathcal{A}_2\big).
\end{equation*}
For later use, we also define
\begin{equation*}
    \begin{cases}
        \Omega_+=\Omega \cap\{y\geq 0\},\\
        \Omega_-=\Omega \cap\{y\leq 0\},\\
        \Sigma_+(M_0)=\Sigma(M_0)\cap\{y\geq 0\},\\
        \Sigma_-(M_0)=\Sigma(M_0)\cap\{y\leq 0\}.
    \end{cases}
\end{equation*}
Furthermore, since
\[
\sin\omega=\frac{c}{\sqrt{u^2+v^2}}<\frac{c}{c_0},\quad \text{for }\sqrt{u^2+v^2}>c_0,
\]
there exists a constant $c_1\in(0,c_0]$, independent of $u_0\ge c_0$ such that
\begin{equation}\label{3.3_10}
    \frac{1+\kappa(\tau)\cos^2 2\omega+ 2\tau\kappa'(\tau)\cos^2 \omega}{2\cos^2 \omega}>1\quad \text{for }0<c<c_1.
\end{equation}
\begin{lemma}\label{3.3_l1}
For $u_0>c_0$, the supersonic jet flow solution of the free boundary problem \eqref{Euler}--\eqref{Boundary_v} satisfies
\begin{equation}\label{3.3_2}
    \frac{\bar\partial^+c}{\cos\omega}>
    -\frac{1}{\mathcal{K}\,R_{\mathcal{A}_2}},
    \qquad
    \frac{\bar\partial^-c}{\cos\omega}>
    -\frac{1}{\mathcal{K}\,R_{\mathcal{A}_1}}
    \quad \text{in } \Sigma(M_0),
\end{equation}
where 
\[
\mathcal{K}
=
\frac{1}{2c_0\mu^2(\tau_0)}.
\]
\end{lemma}
\begin{proof}
We establish the estimate
\begin{equation}\label{3.3_3}
    \frac{\bar\partial^+c}{\cos\omega}>
    -\frac{1}{\mathcal{K}\,R_{\mathcal{A}_2}}
    \quad \text{in } \Sigma(M_0).
\end{equation}
Since $\bar\partial^+c=0$ in $\Sigma(M_0)\setminus(S_-\cup\Omega)$, it is sufficient to verify \eqref{3.3_3} in the region $S_-\cup\Omega$.

Furthermore, by \eqref{ch1_7} and the fact that $\bar\partial^\pm c\leq0$, we obtain
\begin{equation}\label{3.3_5}
\begin{split}
2c\bar\partial^-\omega
&=(1+\kappa(\tau))\tan\omega\bigl(1-\mathcal{Q}\cos^2 \omega\bigr)\bar\partial^-c\\
&=\tan\omega(1+\kappa(\tau)\sin^2 \omega)\bar\partial^-c
<0.
\end{split}
\end{equation}
Substituting \eqref{3.3_5} into the first equation of \eqref{ch2_6}, we deduce
\begin{equation}\label{3.3_6}
\bar\partial^-
\left(
\frac{\cos\omega}{\bar\partial^+c}
\right)
<
-\frac{1}{2\mu^2(\tau)c\cos\omega}
<
-\mathcal{K},
\qquad
\text{in }S_-\cup\Omega.
\end{equation}
Let $\mathcal{P}$ be an arbitrary point in $S_-\cup\Omega$ and let $\widehat{{\mathcal{A}_2}\mathcal{P}}$ denotes the $C_-$ characteristic curve joining ${\mathcal{A}_2}$ to $\mathcal{P}$. By the construction of the simple wave, it follows that
\begin{equation}\label{3.3_7}
    \left(
    \frac{\cos\omega}{\bar\partial^+c}
    \right)(x,y)\longrightarrow 0,
    \qquad\text{as }(x,y)\rightarrow {\mathcal{A}_2}.
\end{equation}
Integrating the differential inequality \eqref{3.3_6} along the characteristic arc $\widehat{{\mathcal{A}_2}\mathcal{P}}$ and using \eqref{3.3_7}, we obtain
\begin{equation*}
    \left(
    \frac{\cos\omega}{\bar\partial^+c}
    \right)(\mathcal{P})
    >
    -\frac{1}{\mathcal{K}\,|\widehat{{\mathcal{A}_2}\mathcal{P}}|}
    >
    -\frac{1}{\mathcal{K}\,\operatorname{dist}({\mathcal{A}_2},\mathcal{P})},
\end{equation*}
where $|\widehat{{\mathcal{A}_2}\mathcal{P}}|$ denotes the arc length of the characteristic segment $\widehat{{\mathcal{A}_2}\mathcal{P}}$. Since the point $\mathcal{P}$ is arbitrary, estimate \eqref{3.3_3} follows throughout $S_-\cup\Omega$. The second inequality in \eqref{3.3_2} can be established by an analogous argument. This completes the proof.
\end{proof}

\begin{lemma}\label{3.3_l2}
For $u_0>c_0$, the supersonic jet flow solution of the free boundary problem \eqref{Euler}--\eqref{Boundary_v} satisfies 
    \begin{align*}
        \begin{split}
            \frac{\bar\partial^+c}{c}&>-\max\left\{\frac{2}{\kappa(\tau_0)R_{\mathcal{A}_2}},\frac{1}{\mathcal{K}lc_1}\right\}, \quad \text{in } \Sigma_+(M_0);\\
            \frac{\bar\partial^-c}{c}&>-\max\left\{\frac{2}{\kappa(\tau_0)R_{\mathcal{A}_1}},\frac{1}{\mathcal{K}lc_1}\right\}, \quad \text{in } \Sigma_-(M_0);\\
            \frac{\bar\partial^\pm c}{c}&>-\max\left\{\frac{2c_0}{\kappa(\tau_0 )lc_1},\frac{c_0}{\mathcal{K}lc_1^2}\right\},\quad \text{in } \Sigma_\mp(M_0).
        \end{split}
    \end{align*}
\end{lemma}
\begin{proof}
    Since the flow remains constant, namely $(c, u, v)=(c_0,u_0,0)$, in the region $\Sigma(M_0)\setminus(S_-\cup S_+\cup\Omega)$, it suffices to establish the desired estimates in $S_-\cup S_+\cup\Omega$.
    
From the characteristic decomposition \eqref{ch2_4}, we obtain
\begin{equation}\label{3.3_12}
\resizebox{\linewidth}{!}{$
    \begin{split}
        \bar\partial^-\left(\frac{c}{\bar\partial^+c}\right)&=-\frac{1}{\left(\frac{c}{\bar\partial^+c}\right)^2}\bar\partial^-\left(\frac{\bar\partial^+c}{c}\right)\\[1mm]
       &=-\frac{1}{c^2\left(\frac{c}{\bar\partial^+c}\right)^2} \left(\frac{\bar\partial^+c\bar\partial^+c}{2\mu^2(\tau)\cos^2 \omega} + \frac{\bar\partial^+c\bar\partial^-c}{2\cos^2 \omega}\left(\frac{1}{\mu^2(\tau)}-\kappa(\tau)\sin^2 2\omega \right) + (\tau\kappa'(\tau)-1)\bar\partial^+c\bar\partial^-c\right)\\[1mm]
         &=-\frac{1}{\left(\frac{c}{\bar\partial^+c}\right)^2}\left(\frac{\kappa(\tau)\bar\partial^+c\bar\partial^+c}{2c^2\cos^2\omega} + \frac{\bar\partial^+c\bar\partial^-c}{c^2}\left(\frac{\kappa(\tau)\cos^2 2\omega+2\sin^2 \omega}{2\cos^2 \omega}+\tau\kappa'(\tau)\right)+\frac{\bar\partial^+c}{2c\cos^2 \omega}\left(\frac{\bar\partial^+c}{c}-\frac{\bar\partial^-c}{c} \right) \right). 
    \end{split}
$}
\end{equation}
Our next objective is to prove that in $S_-\cup \Omega_+$
\begin{equation}\label{3.3_13}
    \bar\partial^-\left(\frac{c}{\bar\partial^+c}\right)<-\frac{2}{\kappa(\tau_0)}\quad \text{if } \frac{\bar\partial^+c}{c}<-\frac{1}{\mathcal{K}lc_1}.
\end{equation}
We first consider the region where $c\geq c_1$. By Lemma \ref{3.3_l1},
\begin{equation*}
    \bar\partial^-c>-\frac{1}{\mathcal{K}l}
    \qquad\text{in }S_-\cup\Omega_+.
\end{equation*}
Therefore,
\[
\frac{\bar\partial^-c}{c}
>
-\frac{1}{\mathcal{K}lc_1},
\qquad
\text{for }y>0.
\]
Hence, if
\[
\frac{\bar\partial^+c}{c}
<
-\frac{1}{\mathcal{K}lc_1},
\]
then
\[
\frac{\bar\partial^+c}{c}
-
\frac{\bar\partial^-c}{c}
<0.
\]
Substituting this into \eqref{3.3_12} yields
\[
\bar\partial^-
\left(
\frac{c}{\bar\partial^+c}
\right)
<
-\frac{2}{\kappa(\tau_0)\cos^2 \omega}
<
-\frac{2}{\kappa(\tau_0)}.
\]
It remains to consider the region where $c<c_1$. In this case, combining \eqref{3.3_10} with \eqref{3.3_12}, we obtain
\[
\bar\partial^-
\left(
\frac{c}{\bar\partial^+c}
\right)
<
-\frac{1}{2\mu^2(\tau)\cos^2 \omega}
<
-\frac{2}{\kappa(\tau_0)},
\]
which establishes \eqref{3.3_13}.

Now let $\mathcal{P}$ be an arbitrary point in $S_-\cup\Omega_+$ and let $\widehat{{\mathcal{A}_2}\mathcal{P}}$ denote the $C_-$ characteristic curve joining ${\mathcal{A}_2}$ to $\mathcal{P}$. It follows that along $\widehat{{\mathcal{A}_2}\mathcal{P}}$
\begin{equation}\label{3.3_15}
    \left(
    \frac{\bar\partial^+c}{c}
    \right)(x,y)
    \longrightarrow
    -\infty,
    \qquad
    \text{as }(x,y)\rightarrow {\mathcal{A}_2}.
\end{equation}
Integrating the differential inequality \eqref{3.3_13} along the characteristic arc $\widehat{{\mathcal{A}_2}\mathcal{P}}$ and using \eqref{3.3_15}, we arrive at
\begin{equation}\label{3.3_16}
    \frac{\bar\partial^+c}{c}>-\max\left\{\frac{2}{\kappa(\tau_0)R_{\mathcal{A}_2}}, \frac{1}{\mathcal{K}lc_1}\right\}.
\end{equation}
Since the point $\mathcal{P}$ is arbitrary, estimate \eqref{3.3_16} holds throughout $\Sigma_+(M_0)$. By symmetry, we also obtain
\[
\frac{\bar\partial^-c}{c}>
-\max\left\{
\frac{2}{\kappa(\tau_0)R_{\mathcal{A}_1}},
\frac{1}{\mathcal{K}lc_1}
\right\}
\qquad\text{in }\Sigma_-(M_0).
\]
It remains to estimate $\displaystyle{\frac{\bar\partial^+c}{c}}$ in $\Sigma_-(M_0)$. By the first characteristic decomposition in Corollary \ref{corollary1}, we have
\begin{equation*}
    \bar\partial^-\left(
    \frac{\bar\partial^+c}{c}
    \right)
    =
    \frac{\kappa(\tau)}{2\cos^2 \omega}
    \left(
    \frac{\bar\partial^+c}{c}
    \right)^2
    +
    \left(
    \frac{\kappa(\tau)+1-\kappa(\tau)\sin^2 2\omega}{2\cos^2 \omega}-1
    \right)
    \frac{\bar\partial^+c}{c}
    \cdot
    \frac{\bar\partial^-c}{c}.
\end{equation*}
Consequently,
\begin{equation}\label{3.3_18}
\begin{cases}
\displaystyle
\bar\partial^-
\left(
\frac{\bar\partial^+c}{c}
\right)
>0,
& c<c_1,\\[2mm]
\displaystyle
\bar\partial^-
\left(
\frac{\bar\partial^+c}{c}
\right)
>
-\frac{\bar\partial^+c}{c}\,
\frac{\bar\partial^-c}{c},
& c\geq c_1.
\end{cases}
\end{equation}
Now let $Q$ be an arbitrary point in $\Omega_-$. The backward $C_-$ characteristic issuing from $Q$ intersects the $x$-axis at a point $Q_-$. In view of the estimate established above, we have
\begin{equation*}
\left(
\frac{\bar\partial^+c}{c}
\right)(Q_-)
\geq
-\max\left\{
\frac{2}{l\kappa(\tau_0)},
\frac{1}{\mathcal{K}lc_1}
\right\}.
\end{equation*}
If $c(Q_-)<c_1$, then the monotonicity property $\bar\partial^-c<0$ implies that $c<c_1$ along the entire characteristic segment $\widehat{Q_-Q}$. Therefore, by \eqref{3.3_18}
\[
\left(\frac{\bar\partial^+c}{c}\right)(Q)
>
\left(\frac{\bar\partial^+c}{c}\right)(Q_-)
>
-\max\left\{
\frac{2}{l\kappa(\tau_0)},
\frac{1}{\mathcal{K}lc_1}
\right\}.
\]
Now suppose that $c(Q_-)> c_1$. If there exists a point $Q_n\in\widehat{Q_-Q}$ satisfying $c(Q_n)=c_1$ which, in view of \eqref{3.3_18}, gives
\begin{equation*}
\begin{split}
\left(\frac{\bar\partial^+c}{c}\right)(Q)
&\geq
\left(\frac{\bar\partial^+c}{c}\right)(Q_n)\\
&>
\frac{c(Q_-)}{c_1}
\left(\frac{\bar\partial^+c}{c}\right)(Q_-)\\
&>
-
\max\left\{
\frac{2c_0}{lc_1\kappa(\tau_0)},
\frac{c_0}{\mathcal{K}lc_1^2}
\right\}.
\end{split}
\end{equation*}
If $c>c_1$ along the entire characteristic segment $\widehat{Q_-Q}$, then it follows from \eqref{3.3_18} that
\begin{equation*}
\begin{split}
\left(\frac{\bar\partial^+c}{c}\right)(Q)
&>
\frac{c(Q_-)}{c(Q)}
\left(\frac{\bar\partial^+c}{c}\right)(Q_-)\\
&>
-\frac{c_0}{c_1}
\max\left\{
\frac{2}{l\kappa(\tau_0)},
\frac{1}{\mathcal{K}lc_1}
\right\}.
\end{split}
\end{equation*}
Therefore,
\[
\left(\frac{\bar\partial^+c}{c}\right)(Q)>
-\frac{c_0}{c_1}
\max\left\{
\frac{2}{l\kappa(\tau_0)},
\frac{1}{\mathcal{K}lc_1}
\right\}.
\]
Since the point $Q$ is arbitrary in $\Omega_-$, we conclude that
\[
\frac{\bar\partial^+c}{c}>
-\frac{c_0}{c_1}
\max\left\{
\frac{2}{l\kappa(\tau_0)},
\frac{1}{\mathcal{K}lc_1}
\right\}
\qquad\text{in }\Sigma_-(M_0).
\]
By symmetry, the corresponding estimate for $\bar\partial^-c$ holds in $\Sigma_+(M_0)$. This completes the proof.
\end{proof}

\subsection{Sonic-supersonic jet solution for the vacuum case}

Next, we turn to the sonic case by considering the free boundary value problem \eqref{Euler}--\eqref{Boundary_v} with $u_0=c_0$. To construct a solution, we approximate the sonic incoming flow by a sequence of supersonic flows.
According to the results established in Subsections \ref{vacuum_subsection1} and \ref{vacuum_subsection2}, for each positive $n$ the free boundary value problem \eqref{Euler}--\eqref{Boundary_v} possesses a global supersonic jet solution $\{(c_n,u_n,v_n)\}_{n=1}^\infty$ corresponding to $u_0=u_0^n:=c_0+\frac{c_0}{n}$.

\begin{lemma}\label{3.4_l1}
Let $\Delta\Subset\Sigma(1)$ be any compact subset. Then there exists a positive integer $N_\Delta$ such that, for every $n>N_\Delta$, the corresponding smooth supersonic jet flow solution  $\{(c_n,u_n,v_n)\}_{n=1}^\infty$
satisfies
\begin{equation}\label{3.4_2}
    (u_0^n)^2<u_n^2+v_n^2<\hat q^{\,2},
    \quad
    [(c_n,u_n,v_n)]_{C^{0,1}(\Delta)}
    <2\kappa(\tau_0)\,\mathcal{M}(\varepsilon),
\end{equation}
where
\begin{equation*}
    \begin{split}
    \varepsilon&=\operatorname{dist}\bigl(\Delta,\{\mathcal{A}_1,\mathcal{A}_2\}\bigr),\\
        \hat q&=\sqrt{4c_0^2+2\left(\frac{K((\gamma+1)\tau_0-b)}{\gamma(\tau_0-b)^{\gamma+1}}-\frac{2a}{\tau_0}\right)},\\
        [(c_n,u_n,v_n)]_{C^{0,1}(\Delta)}
&=
\max\Bigl\{
[c_n]_{C^{0,1}(\Delta)},
[u_n]_{C^{0,1}(\Delta)},
[v_n]_{C^{0,1}(\Delta)}
\Bigr\},\\
\mathcal{M}(\varepsilon)
&=
-\sqrt{2}\,
\max\left\{
\frac{1}{\mathcal{K}\varepsilon},
\frac{2\hat q\,c_0}{\kappa(\tau_0)\varepsilon},
\frac{\hat q}{\mathcal{K}lc_1},
\frac{2\hat q\,c_0}{\kappa(\tau_0)lc_1},
\frac{2\hat q\,c_0}{\kappa(\tau_0)lc_1^2\mathcal{K}}
\right\}.
    \end{split}
\end{equation*}

\end{lemma}
\begin{proof}
Since $\bar\partial^\pm c_n<0$, it follows that
\begin{equation*}
    \begin{split}
        \frac{\bar\partial^\pm c_n}{\sin\omega_n\cos\omega_n}
    &>
    \sqrt{2}\,
    \min\left\{
    \frac{\bar\partial^\pm c_n}{\cos\omega_n},
    \frac{\bar\partial^\pm c_n}{\sin\omega_n}
    \right\}\\
   & =
    \sqrt{2}\,
    \min\left\{
    \frac{\bar\partial^\pm c_n}{\cos\omega_n},
    \frac{\sqrt{u_n^2+v_n^2}\,\bar\partial^\pm c_n}{c_n}
    \right\}.
    \end{split}
\end{equation*}
Moreover, the Bernoulli relation implies that
\[
u_0^n<\sqrt{u_n^2+v_n^2}<\hat q
\qquad
\text{in }
\Sigma\!\left(1+\frac{1}{n}\right).
\]
Therefore, applying Lemmas \ref{3.3_l1} and \ref{3.3_l2}, we obtain
\begin{equation}\label{3.4_4}
    \frac{\bar\partial^\pm c_n}{\sin\omega_n\cos\omega_n}
    >
   \mathcal{M}(\varepsilon)
    \qquad
    \text{in }\Delta.
\end{equation}
Finally, combining \eqref{3.4_4} with \eqref{derivatives} and \eqref{ch1_8} yields the estimate \eqref{3.4_2}. The proof is completed.
\end{proof}
By Lemma \ref{3.4_l1}, together with a standard diagonal argument, there is a bounded function
\[
(\breve c,\breve u,\breve v)\in C^{0,1}\bigl(\Sigma(1)\bigr)
\]
together with a subsequence
\[
\{(c_{n_j},u_{n_j},v_{n_j})\}_{j=1}^{\infty}
\subset
\{(c_n,u_n,v_n)\}_{n=1}^{\infty},
\]
satisfying, for every compact subset $\Delta\Subset\Sigma(1)$,
\begin{equation}\label{3.4_5}
    (c_{n_j},u_{n_j},v_{n_j})
    \longrightarrow
    (\breve c,\breve u,\breve v)
    \qquad
    \text{in }C^{0,1}(\Delta)
    \quad\text{as }j\to\infty.
\end{equation}

\begin{lemma}\label{3.4_l2}
The limit function $(\breve c,\breve u,\breve v)$ obtained in \eqref{3.4_5} extends continuously to $\overline{\Sigma(1)}$. Furthermore,
\[
(\breve c,\breve u,\breve v)=(c_0,c_0, 0)
\qquad\text{on }\overline{\mathcal{A}_1\mathcal{A}_2},
\]
and
\[
(\breve c,\breve u,\breve v)|_{x=\varphi(y;1),\,y<-l}=(0, u_{vac},v_{vac}),\quad (\breve c,\breve u,\breve v)|_{x=\varphi(y;1),\,y>l}=(0,u_{vac},-v_{vac}),
\]
where $\displaystyle{u_{vac}=\hat u(\alpha_v)\;\text{and} \;v_{vac}=\hat v(\alpha_v),}$ given by \eqref{Simple_1}.
\end{lemma}

\begin{proof}
Fix $\varepsilon\in(0,l)$ and let $\eta>0$ be sufficiently small. Define
\begin{equation*}
    \mathcal{X}(\varepsilon,\eta)
    =
    \left\{
    (x,y)\; : \;
    x\in (0,\eta),\;
    |y|<l-\varepsilon
    \right\}.
\end{equation*}
By Lemma \ref{3.4_l1}, we have
\begin{equation*}
    [(c_{n_j},u_{n_j},v_{n_j})]_{C^{0,1}(\mathcal{X}(\varepsilon,\eta))}
    <
    2\kappa(\tau_0)\,\mathcal{M}(\varepsilon).
\end{equation*}
Consequently,
\begin{equation}\label{3.4_8}
\begin{split}
    \|u_{n_j}-u_0^{\,n_j}\|_{0;\mathcal{X}(\varepsilon,\eta)}
    &<
    2\kappa(\tau_0)\,\mathcal{M}(\varepsilon)\eta,\\
    \|v_{n_j}\|_{0;\mathcal{X}(\varepsilon,\eta)}
    &<
    2\kappa(\tau_0)\,\mathcal{M}(\varepsilon)\eta,\\
    \|c_{n_j}-c_0\|_{0;\mathcal{X}(\varepsilon,\eta)}
    &<
    2\kappa(\tau_0)\,\mathcal{M}(\varepsilon)\eta.
\end{split}
\end{equation}
Letting $j\to\infty$ in \eqref{3.4_8} yields
\begin{equation*}
    \|(\breve c,\breve u,\breve v)-(c_0,c_0, 0)\|_{0;\mathcal{X}(\varepsilon,\eta)}
    <
    6\kappa(\tau_0)\,\mathcal{M}(\varepsilon)\eta.
\end{equation*}
Since $\eta>0$ is arbitrary, it follows that, for every $y_0\in(-l,l)$
\begin{equation*}
    (\breve c,\breve u,\breve v)(x,y)
    \longrightarrow
    (c_0,c_0, 0),
    \qquad
    \text{as }(x,y)\rightarrow(0,y_0).
\end{equation*}
For any sufficiently small $\varepsilon>0$ and $\eta>0$, define
\begin{equation*}
    \mathcal{Y}(\varepsilon,\eta)
    =
    \left\{
    (x,y)\;\middle|\;
    \varphi(y;l)<x<\varphi(y;l)+\eta,\;
    y<-l-\varepsilon
    \right\}.
\end{equation*}
Let $(x_0,y_0)\in\mathcal{Y}(\varepsilon,\eta)$. Then, for all sufficiently large $j$, we have
\[
(x_0,y_0)\in\Sigma\!\left(1+\frac{1}{n_j}\right).
\]
Applying Lemma \ref{3.4_l1}, we obtain
\begin{equation*}
[(c_{n_j},u_{n_j},v_{n_j})]_{C^{0,1}\!\left(
\mathcal{Y}(\varepsilon,\eta)\cap
\Sigma\!\left(1+\frac{1}{n_j}\right)
\right)}
<
2\kappa(\tau_0)\,\mathcal{M}(\varepsilon).
\end{equation*}
Moreover, since
\[
\varphi(y;l)<\varphi\!\left(y;1+\frac{1}{n_j}\right),
\qquad y<-l,
\]
it follows that
\begin{equation*}
\begin{split}
\left|u_{n_j}(x_0,y_0)-u_{v,n_j}\right|
&<
2\kappa(\tau_0)\,\mathcal{M}(\varepsilon)\eta,\\
\left|v_{n_j}(x_0,y_0)-v_{v,n_j}\right|
&<
2\kappa(\tau_0)\,\mathcal{M}(\varepsilon)\eta,\\
|c_{n_j}(x_0,y_0)|
&<
2\kappa(\tau_0)\,\mathcal{M}(\varepsilon)\eta,
\end{split}
\end{equation*}
where $(u_{v,n_j},v_{v,n_j})$ denotes the vacuum state corresponding to the simple wave solution associated with $u_0=u_0^{n_j}$. By letting $j\to\infty$, we arrive at
\begin{equation*}
    |(\breve c,\breve u,\breve v)(x_0,y_0)-(0,u_{vac}, v_{vac})|<6\kappa(\tau_0)\,\mathcal{M}(\varepsilon)\eta.
\end{equation*}
Consequently,
\begin{equation*}
\|
(\breve c,\breve u,\breve v)
-
(0,u_{vac},v_{vac})
\|_{0;\mathcal{Y}(\varepsilon,\eta)}
\leq
6\kappa(\tau_0)\,\mathcal{M}(\varepsilon)\eta.
\end{equation*}
\begin{figure}
    \centering
    \includegraphics[width=0.5\linewidth]{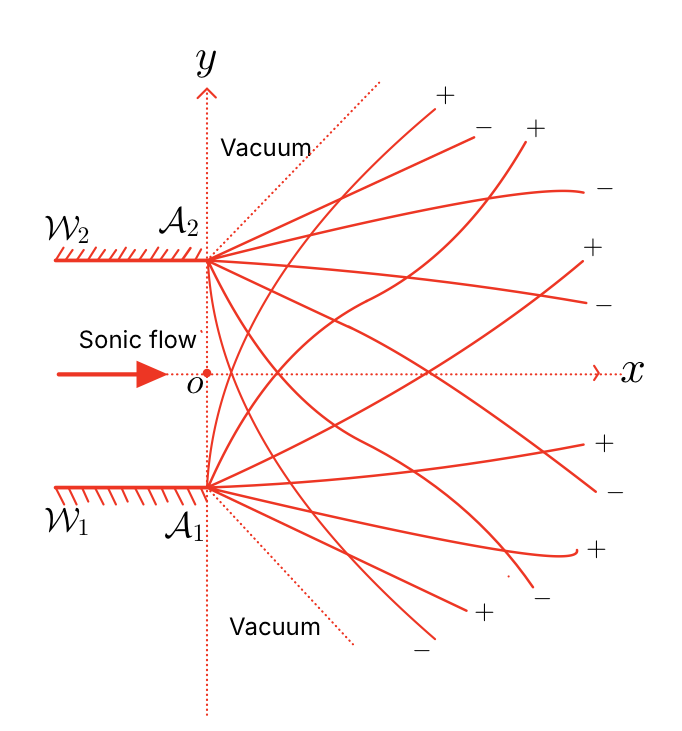}
    \caption {Configuration of the sonic--supersonic jet expanding into a vacuum: The wave pattern together with the symmetric vacuum boundaries emanating from the nozzle exit.}
    \label{Vacuum_flow}
\end{figure}
In view of the arbitrariness of $\varepsilon$ and $\eta$, $(\breve c,\breve u,\breve v)$ extends continuously to the vacuum boundary
\[
\{(x,y): x=\varphi(y;1),\;y<-l\},
\]
and satisfies
\[
(\breve c,\breve u,\breve v)(\varphi(y;1),y)
=
(0,u_{vac},v_{vac}),
\qquad
y<-l.
\]
Using symmetry, $(\breve c,\breve u,\breve v)$ also extends continuously to
\[
\{(x,y)\mid x=\varphi(y;1),\;y>l\},
\]
with boundary value
\[
(\breve c,\breve u,\breve v)(\varphi(y;1),y)
=
(0,u_{vac},-v_{vac}),
\qquad
y>l.
\]
This concludes the proof.
\end{proof}
Combining the results established above, we conclude the existence of a globally defined jet solution for sonic-supersonic flow with Lipschitz continuity in the vacuum case from the straight two-dimensional nozzle. The corresponding wave structure is depicted in Figure \ref{Vacuum_flow}. The proof of Theorem \ref{main1} for $u_0=c_0$ is therefore completed.

\section{Sonic-supersonic jet flows from the straight 2-D nozzle into the atmosphere}

The present section investigates
\eqref{Euler}--\eqref{Boundary_p}. Let $p_{atm}=p(\tau_{atm})$ denote the prescribed
atmospheric pressure, where $\tau_{atm}$ is the corresponding specific volume.
The associated sound speed is given by
\begin{equation*}
c_{atm}=
\sqrt{\frac{K(\gamma+1)\tau_{atm}^2}{(\tau_{atm}-b)^{\gamma+2}}
-\frac{2a}{\tau_{atm}}}.
\end{equation*}
Let $\alpha_{atm}$ be the unique angle satisfying
$\hat c(\alpha_{atm})=c_{atm}$. Furthermore, define
\begin{equation*}
\theta_{atm}
=
\alpha_{atm}-
\arcsin\!\left(
\frac{c_{atm}}{q_{atm}}
\right),
\end{equation*}
where $q_{atm}$ is determined from the Bernoulli relation
\begin{equation*}
q_{atm}=
\sqrt{
u_0^2
+\frac{2K((\gamma+1)\tau_0-b)}{\gamma(\tau_0-b)^{\gamma+1}}
-\frac{4a}{\tau_0}
-\frac{2K((\gamma+1)\tau_{atm}-b)}{\gamma(\tau_{atm}-b)^{\gamma+1}}
+\frac{4a}{\tau_{atm}}
}.
\end{equation*}
Throughout the present section, we suppose that
\begin{equation}\label{4_3}
    -\pi<\theta_{atm}<0.
\end{equation}

\subsection{Local supersonic jet flow into an atmosphere}

Throughout this subsection, we assume that \eqref{4_3} holds. The uniform supersonic flow issuing from the two corners of the nozzle expands through a symmetric pair of centered simple waves, denoted by $\mathcal{S}_+$ and $\mathcal{S}_-$, until $p=p_{atm}$ is attained. The resulting jet is separated from the surrounding stationary atmosphere by two free boundaries, each of which is a contact discontinuity.

According to the centered simple wave solution established in Theorem \ref{simplewaveth}, the simple waves $\mathcal{S}_\mp$ are described by
\begin{equation*}
    \begin{split}
    &\mathcal{S}_-:(c, u, v)
    =
    (c_1, u_1, v_1)(x,y),
    \qquad
    \alpha_{atm}<
    \arctan\!\left(\frac{l+ y}{x}\right)
    <
    \alpha_0\\
        &\mathcal{S}_+:(c, u, v)
    =
    (c_2, u_2, v_2)(x,y),
    \qquad
    \alpha_{atm}<
    \arctan\!\left(\frac{l- y}{x}\right)
    <
    \alpha_0.
    \end{split}
\end{equation*}
The simple wave $\mathcal{S}_+$ is enclosed by the two characteristic $C_+$ straight lines
\[
\overline{\mathcal{A}_1P_0}: \; y=x\tan\alpha_0-l,
\qquad
\overline{\mathcal{A}_1\mathcal{B}_1}: \; y=x\tan\alpha_{atm}-l.
\]
The region beyond the characteristic line $\overline{\mathcal{A}_1\mathcal{B}_1}$ is occupied by a constant state, which is parted from the surrounding atmosphere by the jet boundary
\[
J_-:\; \frac{y+l}{x}=\tan\theta_{atm}.
\]
Similarly, $\mathcal{S}_-$ is enclosed by the two characteristic $C_-$ straight lines
\[
\overline{\mathcal{A}_2P_0}: \; y=l-x\tan\alpha_0,
\qquad
\overline{\mathcal{A}_2E}: \; y=l-x\tan\alpha_{atm}.
\]
The constant-state region behind the characteristic line $\overline{\mathcal{A}_2E}$ is bounded by the jet boundary
\[
J_+:\;\frac{l-y}{x}=\tan\theta_{atm},
\]
which separates the jet from the stationary atmosphere.
The centered simple waves begin to interact at the point
\[
P_0=(l\cot\alpha_0,\,0).
\]
From $P_0$ we construct a $C_+$ ($C_-$, respectively) cross-characteristic curve, denoted by $\widehat{P_0E}$ ($\widehat{P_0\mathcal{B}_1}$, respectively), lying in the simple wave $\mathcal{S}_-$ ($\mathcal{S}_+$, respectively). The flow in the interaction region is then determined by solving the Goursat problem for \eqref{Euler} subject to the boundary condition
\begin{equation}\label{4.1_3}
\left.(c, u, v)\right|_{\widehat{P_0\mathcal{B}_1}}
=(c_2,u_2,v_2),\qquad
\left.(c, u, v)\right|_{\widehat{P_0E}}
=(c_1,u_1,v_1).
\end{equation}
By the results established in Section \ref{section_3}, \eqref{Euler}-\eqref{4.1_3} possesses a global solution in the domain enclosed by the characteristics $\widehat{P_0E}$, $\widehat{P_0\mathcal{B}_1}$, positive characteristic $C_+^{\mathcal{B}_1}$ emanating from $\mathcal{B}_1$ and the negative characteristic $C_-^E$ emanating from $E$. Furthermore, the solution obeys the estimates obtained in Lemma \ref{3.2_l2}.

Since the flow adjacent to a constant-state region is itself a simple wave, two non-planar simple waves, denoted by $\mathcal{R}_+$ and $\mathcal{R}_-$, are generated from the characteristic curves $C_-^E$ and $C_+^{\mathcal{B}_1}$, respectively. The characteristic curves of $\mathcal{R}_+$ ($\mathcal{R}_-$, respectively) are $C_+$ ($C_-$, respectively) straight characteristics.

Lemma~\ref{3.2_l2} yields
\[
\bar\partial^+ c|_{C_+^{\mathcal{B}_1}}<0.
\]
Exploiting the above inequality together with \eqref{ch1_6} and the characteristic decomposition \eqref{ch2_2}, it follows that the simple wave $\mathcal{R}_-$ obeys
\[
\bar\partial^+\beta>0, \quad
\bar\partial^+c<0.
\]
Consequently, the straight characteristic $C_-$ originating from $C_+^{\mathcal{B}_1}$ remains non-intersecting.

Let $\mathcal{B}_2$ ($H$, respectively) denote the intersection point of the straight $C_-$ ($C_+$, respectively) characteristic emanating from $\mathcal{B}_1$ ($E$, respectively) with the jet boundary $J_-$ ($J_+$, respectively). The simple waves $\mathcal{R}_+$ and $\mathcal{R}_-$ are then reflected from the jet boundary at the points $\mathcal{B}_2$ and $H$, respectively.

Finally, we construct a $C_+$ ($C_-$, respectively) cross characteristic, denoted by $C_+^{\mathcal{B}_2}$ ($C_-^H$, respectively), inside the simple wave $\mathcal{R}_+$ ($\mathcal{R}_-$, respectively). Since $\bar\partial^+c<0$ on $C_+^{\mathcal{B}_1}$, another application of \eqref{ch2_2} yields
\[
\bar\partial^+c|_{C_+^{\mathcal{B}_2}}<0.
\]
To investigate the reflected waves families $\mathcal{R}_+$ and $\mathcal{R}_-$ from the respective jet boundaries $J_-$ and $J_+$, we analyze \eqref{Euler} subject to the boundary conditions
\begin{equation}\label{4.1_4}
    \begin{cases}
        (c, u, v)=(c^1, u^1, v^1)(x,y),
        & (x,y)\in C_+^{\mathcal{B}_2},\\[1mm]
        v(x_0,y_0)(s)x_0'(s)
        =
        u(x_0,y_0)(s)y_0'(s),\qquad
        x_0'^2+y_0'^2=1,
        & s>0,\\[1mm]
        p(x_0,y_0)(s)=p_{atm},
        & s>0,\\[1mm]
        (x_0(0),y_0(0))=(x_{\mathcal{B}_2},y_{\mathcal{B}_2}),
    \end{cases}
\end{equation}
where $(c^1, u^1, v^1)(x,y)$ denotes the flow state prescribed along the characteristic curve $C_+^{\mathcal{B}_2}$.

The free boundary is the unknown jet boundary
\[
J_-:\;(x,y)=(x_0,y_0)(s),\qquad s>0.
\]
For later convenience, we parameterize the characteristic curve $C_+^{\mathcal{B}_2}$ by
\[
x=\tilde{x}_{\mathcal{B}_2}(s),\qquad
y=\tilde{y}_{\mathcal{B}_2}(s),
\]
where $s\ge0$ denotes the arc-length parameter. We now establish a local existence result for the free BVP \eqref{Euler}--\eqref{4.1_4}.
\begin{lemma}\label{4.1_l1}
There exists a sufficiently small constant $\epsilon_0 >0$ such that the free boundary value problem \eqref{Euler}--\eqref{4.1_4} admits a unique classical solution in the domain $\Phi_{\epsilon_0 }$ bounded by the characteristic curve $C_+^{\mathcal{B}_2}$, the jet boundary $J_-$ and the $C_-$ characteristic originating from the point $\bigl(\tilde{x}_F(\epsilon_0 ),\tilde{y}_F(\epsilon_0 )\bigr)$. Here, the constant $\epsilon_0 $ depends only on the $C^1$-norm of the prescribed boundary condition on $C_+^{\mathcal{B}_2}$. Furthermore, the solution obeys
\begin{equation}\label{4.1_5}
    \bar\partial^-c>0, \quad \bar\partial_0\theta\big|_{J_-}>0,
    \quad
    \bar\partial^+c<0.
\end{equation}
\end{lemma}

\begin{proof}
The local solvability of the free boundary problem \eqref{Euler}--\eqref{4.1_4} follows directly from the local existence theorem established in \cite{chen2010}. It therefore remains only to verify the estimates in \eqref{4.1_5}.

From the boundary condition on $C_+^{\mathcal{B}_2}$, we have
\[
\bar\partial^+c|_{C_+^{\mathcal{B}_2}}<0.
\]
An application of the first characteristic decomposition in Corollary \ref{corollary1} shows that
\begin{equation*}
    \bar\partial^+c<0
    \qquad\text{in }\Phi_{\epsilon_0 }.
\end{equation*}
The boundary condition in \eqref{4.1_4} further implies that
\begin{equation}\label{4.1_7}
    \bar\partial_0c=\frac{(\bar\partial_++\bar\partial_-)c}{2}=0
    \qquad\text{on }J_-.
\end{equation}
Consequently,
\[
\bar\partial^-c|_{J_-}>0,
\]
and the second characteristic decomposition in Corollary \ref{corollary1} yields
\begin{equation*}
    \bar\partial^-c>0
    \qquad\text{in }\Phi_{\epsilon_0 }.
\end{equation*}
Using \eqref{conditions}, \eqref{ch1_6} together with \eqref{ch1_7}, a straightforward calculation gives
\begin{equation*}
   \begin{split}
    \bar\partial_0\theta&=\left(\frac{\bar\partial^++\bar\partial^-}{4\cos\omega}\right)(\alpha+\beta)\\&=\left(\frac{\kappa(\tau)}{2\sqrt{u^2+v^2}}\right)(\bar\partial^--\bar\partial^+)c\\
       \implies  \bar\partial_0\theta&>0,
    \qquad \text{on } J_-.
   \end{split}
\end{equation*}
The proof is therefore completed.
\end{proof}


\subsection{Local sonic-supersonic jet flow solution in an atmosphere}
We next turn our attention to the construction of a local sonic--supersonic jet solution in the presence of an atmosphere. Let $\alpha_{atm}^{\,n_j}$ and $\alpha_{atm}^\mathcal{L}$ denote the values of $\alpha$ satisfying
\[
\hat c_{\,n_j}(\alpha)=c_{atm}
\quad\text{and}\quad
\hat c(\alpha)=c_{atm},
\]
respectively. To this end, we introduce the following quantities:
\begin{equation*}
\begin{split}
\begin{cases}
\displaystyle
    q_{atm}=\sqrt{3c_0^2-2c_{atm}^2+\frac{K((\tau_0-b)^2-\gamma\tau_0(b+\gamma\tau_0))}{\gamma(\tau_0-b)^{\gamma+2}}-\frac{K((\tau_{atm}-b)^2-\gamma\tau_{atm}(b+\gamma\tau_{atm}))}{\gamma(\tau_{atm}-b)^{\gamma+2}}}\\[1mm]
    \displaystyle
\theta_{atm}^\mathcal{L}
=
\alpha_{atm}^\mathcal{L}-\omega_{atm}^\mathcal{L}=\alpha_{atm}^\mathcal{L}-\arcsin\!\left(
\frac{c_{atm}}{q_{atm}}
\right),\\[1mm]
\beta_{atm}^\mathcal{L}
=
\alpha_{atm}^\mathcal{L}-2\omega_{atm}^\mathcal{L}.
\end{cases}
\end{split}
\end{equation*}
For each $n_j$, let $C_+^{\,n_j}$ denote the positive characteristic curve issuing from the point $\mathcal{A}_1$ with initial slope $\tan\alpha_{atm}^{\,n_j}$ in the corresponding supersonic jet solution for the vacuum case. By the results of Section~\ref{vacuum_subsection2}, for $u_0=u_0^{nj}$ the flow is described by
\[
(c, u, v)=(c_{n_j},u_{n_j},v_{n_j})(x,y).
\]
Parameterizing the characteristic curve $C_+^{\,n_j}$ by its arc length,
\[
x=x_{n_j}(s),\qquad
y=y_{n_j}(s),\qquad s>0,
\]
we have
\begin{equation*}
\begin{cases}
\sqrt{(x_{n_j}'(s))^2+(y_{n_j}'(s))^2}=1,\\[2mm]
\displaystyle
\frac{y_{n_j}'(s)}{x_{n_j}'(s)}
=
\Lambda_+\!\left(
(c_{n_j},u_{n_j},v_{n_j})
(x_{n_j}(s),y_{n_j}(s))
\right),
\end{cases}
\qquad s>0.
\end{equation*}
We then consider the free boundary value problem for \eqref{Euler} subject to 
\begin{equation}\label{4.2_3}
    \begin{cases}
    \displaystyle
        (c, u, v)=(c_{n_j},u_{n_j},v_{n_j})(x,y), & (x,y)\in C_+^{\,n_j};\\[1mm]
        \displaystyle
        u(x,y)=v(x,y)\psi_{n_j}'(y), p(x,y)=p_{atm}, & (x,y)\in \{(x,y)| x=\psi_{n_j}(y),\qquad y<-l\};\\
         \psi_{n_j}(-l)=0.
    \end{cases}
\end{equation}
The unknown jet-free boundary is given by
\begin{equation*}
    x=\psi_{n_j}(y),\qquad y<-l.
\end{equation*}
By Lemma \ref{3.3_l1}, the solution satisfies
\begin{equation}\label{4.2_4}
    -\frac{1}{\mathcal{K}l}
    <
    \bar\partial^+c_{n_j}
    \leq
    0,
    \qquad
    \text{on }
    C_+^{\,n_j}\cap\{y<0\}.
\end{equation}
Throughout this subsection, we assume that
\begin{equation*}
    \beta_{atm}^\mathcal{L}>-\pi.
\end{equation*}
Accordingly, we can find small positive constant $\varepsilon_0>0$ for which
\begin{equation*}
-\pi
<
\beta_{atm}^\mathcal{L}-\varepsilon_0
<
\beta_{atm}^\mathcal{L}+\varepsilon_0
<
\theta_{atm}^\mathcal{L}-\varepsilon_0
<
\theta_{atm}^\mathcal{L}+\varepsilon_0
<
0.
\end{equation*}
By Lemma \ref{4.1_l1}, together with \eqref{4.1_7} and \eqref{4.2_4}, there exists a uniform constant $\epsilon_0>0$, such that for all sufficiently large $j$, we can establish the piecewise smooth unique solution for the problem \eqref{Euler}, \eqref{4.2_3}, namely 
\[
({c}^*_{n_j},{u}^*_{n_j},{v}^*_{n_j})(x,y)
\]
in the domain $\Psi_-^{\,n_j}$ bounded by the characteristic curve $C_+^{\,n_j}$, the free boundary $x=\psi_{n_j}(y)$ and the $C_-$ characteristic issuing from the point $\bigl(x_{n_j}(\epsilon_0 ),y_{n_j}(\epsilon_0 )\bigr)$. Furthermore, the solution admits
\begin{equation*}
\left|(\theta,\alpha,\beta)
-(\theta_{atm}^\mathcal{L},\alpha_{atm}^\mathcal{L},\beta_{atm}^\mathcal{L})\right|
\leq \varepsilon_0,
\qquad
\frac{c_{atm}}{2}\leq  c^*_{n_j}\leq c_{atm},
\end{equation*}
and
\begin{equation*}
-\frac{2}{\mathcal{K}l}
<
\bar\partial^+c^*_{n_j}
\leq
0,
\qquad
0
\leq
\bar\partial^- c^*_{n_j}
<
\frac{2}{\mathcal{K}l}.
\end{equation*}
Let $x=x_{n_j}^0(y)$ denote the $C_-$ characteristic curve issuing from the point
\[
\bigl(x_{n_j}(\epsilon_0 ),\,y_{n_j}(\epsilon_0 )\bigr).
\]
This characteristic intersects the free boundary $x=\psi_{n_j}(y)$ at the point $\bigl(\psi_{n_j}(y_j),\,y_j\bigr)$. Along the characteristic curve, we have
\begin{equation*}
    \frac{dx_{n_j}^0(y)}{dy}
    =
    \cot\tilde{\beta}_{n_j}
    \bigl(x_{n_j}^0(y),y\bigr),
    \qquad
    y_j<y<y_{n_j}(\epsilon_0 ).
\end{equation*}
On the other hand, the free boundary \eqref{4.2_3} satisfies
\begin{equation*}
    \frac{d\psi_{n_j}(y)}{dy}
    =
    \cot{\theta^*}_{n_j}
    \bigl(\psi_{n_j}(y),y\bigr),
    \qquad
    y_j<y<-l.
\end{equation*}
Thus, applying Arzelà--Ascoli theorem together with a standard diagonal argument, there exists a constant $\xi>0$, a subsequence $\{n_{j_k}\}_{k=1}^{\infty}\subset\{n_j\}_{j=1}^{\infty}$, functions
\[
x^\mathcal{L}(s),\,y^\mathcal{L}(s)\in C[0,\epsilon_0 ],\qquad
x^\mathcal{L}_{0}(y)\in C[-l-\xi,\,y^\mathcal{L}(\epsilon_0 )],
\]
\[
\psi^\mathcal{L}(y)\in C[-l-\xi,\,-l],
\qquad
(u^*, v^*, c^*)\in C^{0,1}(\Psi_-),
\]
such that
\begin{equation*}
    y_{j_k}\longrightarrow -l-\xi,
    \qquad\text{as }k\to\infty,
\end{equation*}
and
\begin{equation*}
\begin{cases}
(x_{n_{j_k}},y_{n_{j_k}})
\longrightarrow
(x^\mathcal{L},y^\mathcal{L})
&\text{in }C[0,\epsilon_0 ],\\[1mm]
x_{n_{j_k}}^{0}
\longrightarrow
x^\mathcal{L}_{0}
&\text{in }C[-l-\xi,\,y^\mathcal{L}(\epsilon_0 )],\\[1mm]
\psi_{n_{j_k}}
\longrightarrow
\psi^\mathcal{L}
&\text{in }C[-l-\xi,\,-l],\\[1mm]
(c^*_{n_{j_k}}, u^*_{n_{j_k}},v^*_{n_{j_k}})
\longrightarrow
(c^*, u^*, v^*)
&\text{in }C^{0,1}(\Psi_-),
\end{cases}
\qquad\text{as }k\to\infty.
\end{equation*}
Here, $\Psi_-$ denotes the domain enclosed by
\[
x= x^\mathcal{L}(s),\quad y= y^\mathcal{L}(s),
\qquad  s\in(0, \epsilon_0),
\]
the characteristic curve
\[
x=x^\mathcal{L}_{0}(y),
\qquad
y\in (-l-\xi, y^\mathcal{L}(\epsilon_0 )),
\]
and the free boundary
\[
x=\psi^\mathcal{L}(y),
\qquad
y\in(-l-\xi,-l).
\]
Using symmetry, an analogous solution is obtained in the domain $\Psi_+$, which is enclosed by
\[
x=x^\mathcal{L}(s),\quad y=-y^\mathcal{L}(s),
\qquad s\in(0, \epsilon_0),
\]
the characteristic curve
\[
x=x^\mathcal{L}_{0}(-y),
\qquad
y\in(-y^\mathcal{L}(\epsilon_0 ), l+\xi),
\]
and the free boundary
\[
x=\psi(-y),
\qquad
y\in (l, l+\xi).
\]
Define
\[
{\mathcal{I}_1}=\bigl(x^\mathcal{L}(\epsilon_0 ),\,y^\mathcal{L}(\epsilon_0 )\bigr),
\qquad
{\mathcal{I}_2}=\bigl(x^\mathcal{L}(\epsilon_0 ),\,-y^\mathcal{L}(\epsilon_0 )\bigr).
\]
Denote by $\widehat{\mathcal{A}_2\mathcal{I}_1}$ and $\widehat{\mathcal{A}_1\mathcal{I}_2}$ the negative and positive characteristic curves connecting $\mathcal{A}_2$ to $\mathcal{I}_1$ and $\mathcal{A}_1$ to $\mathcal{I}_2$, respectively, in the sonic--supersonic jet flow expanding into a vacuum. These two characteristic curves intersect at a point, denoted by $G$.

Finally, let $\Psi_0$ denote the region enclosed by the segment $\overline{\mathcal{A}_1\mathcal{A}_2}$ together with the characteristic arcs $\widehat{E\mathcal{I}_2}$, $\widehat{E\mathcal{I}_1}$, $\widehat{\mathcal{A}_2\mathcal{I}_2}$ and $\widehat{\mathcal{A}_1\mathcal{I}_1}$; see Figure~\ref{Pressure_flow}.

Further define
\begin{equation}\label{4.2_13}
    \begin{split}
        (c, u, v)=& (c^*,u^*,v^*)\quad\text{for }(x,y)\in\Psi_-,\\
        (c, u, v)=&(\breve c,\breve u,\breve v)\quad\text{for }(x,y)\in\Psi_0.
    \end{split}
\end{equation}
The solution in $\Psi_+$ is then obtained by symmetry, namely,
\begin{equation}\label{4.2_14}
    (c, u, v)=\bigl(c^*,u^*,v^*\bigr)(x,-y),\quad (x,y)\in\Psi_+.
\end{equation}
Here, $(\breve c,\breve u,\breve v)$ denotes the sonic--supersonic jet solution corresponding to the vacuum case constructed in Section~\ref{vacuum_subsection2}.

By construction, the function $(c, u, v)$ defined by \eqref{4.2_13}-\eqref{4.2_14} provides a local sonic--supersonic solution to the free BVP \eqref{Euler}--\eqref{Boundary_p} with $u_0=c_0$. Consequently, there exists a sufficiently small constant $\delta>0$ such that the solution is defined in the domain $\Gamma(\delta)$. This establishes Theorem~\ref{main2}.
\begin{figure}
    \centering
    \includegraphics[width=0.5\linewidth]{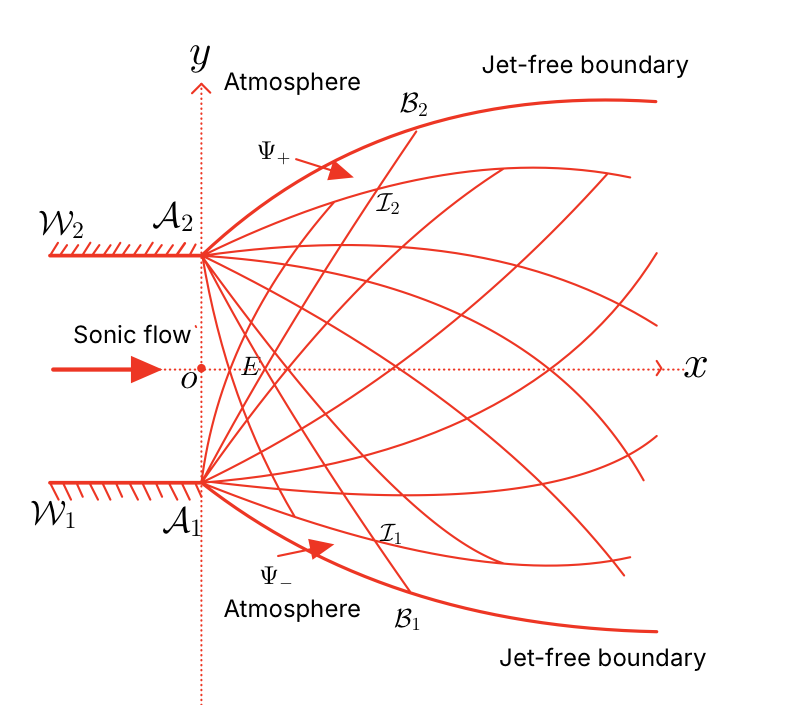}
    \caption {Configuration of the local sonic--supersonic jet flow into an atmosphere from a straight 2-D nozzle and the wave structure together with the free jet boundaries separating the jet from the surrounding atmosphere.}
    \label{Pressure_flow}
\end{figure}


\section{Conclusions}

We investigated sonic-supersonic jet flows issuing from a straight two-dimensional nozzle for compressible fluids governed by the van der Waals equation of state. Owing to the presence of sonic states at the nozzle exit the governing Euler system becomes degenerate hyperbolic which led to a very interesting free boundary value problem. By employing the method of characteristic decomposition together with suitable a priori estimates, we established the global existence of locally Lipschitz continuous sonic-supersonic jet flow solutions expanding into a vacuum. Furthermore, by constructing a family of local supersonic jet flow solutions and passing to the sonic limit through a compactness argument, we proved the local existence of sonic-supersonic jet flow solutions exhausting into a static atmosphere with pressure $0<p_{atm}<p_0$. 

In this work, we extended the existing theory of sonic-supersonic jet flows from the classical polytropic framework to the physically more realistic van der Waals gas model. Although the analysis is carried out under a suitable convexity assumption on the pressure law, many of the analytical techniques developed here rely only on general structural properties of the equation of state and are therefore expected to be applicable to a broader class of convex pressure laws. The results obtained in this article provide a useful foundation for the mathematical analysis of transonic free boundary problems arising in more general nozzle geometries and under different ambient flow conditions.

\section*{Acknowledgements}
\textit{The first author acknowledges the financial support provided by the Indian Institute of Technology Kharagpur. The second author (TRS) expresses his gratitude towards ANRF, DST, Government of India (Ref. No. CRG/2022/006297) for its financial support through the core research grant.}


\end{document}